\documentclass[reqno]{amsart}
\usepackage{amsmath}
  \usepackage{paralist}
  \usepackage{graphics} 
    \usepackage{epsfig} 
\usepackage{graphicx}  \usepackage{epstopdf}
 \usepackage[colorlinks=true]{hyperref}
\usepackage{amsfonts}
\usepackage{mathrsfs}
\usepackage {amssymb,euscript}
\usepackage {amsthm}
\usepackage {amscd}
\usepackage{geometry}
\usepackage{hyperref}
\usepackage{color}

\usepackage{tikz}

\usepackage{bbm}
\usepackage{graphics}
\usepackage[pagewise]{lineno}

\hypersetup{urlcolor=blue, citecolor=red}

\newtheorem{theorem}{Theorem}[section]

\newtheorem{lemma}[theorem]{Lemma}
\newtheorem{proposition}{Proposition}

\theoremstyle{definition}

\newtheorem{remark}{Remark}

\title[Fermi's Golden Rule and $H^1$ Scattering]{Fermi's Golden Rule and $H^1$ Scattering for Nonlinear Klein-Gordon Equations with Metastable States}



 \email{matax@nus.edu.sg}
 \email{soffer@math.rutgers.edu}

\begin{document}
\maketitle

\centerline{\scshape Xinliang An}
\medskip
{\footnotesize

 \centerline{Department of Mathematics, National University of Singapore}

   \centerline{Singapore 119076}
} 

\medskip

\centerline{\scshape Avy Soffer}
\medskip
{\footnotesize

 \centerline{Department of Mathematics, Rutgers University}
   
   \centerline{Piscataway, NJ, USA 08854}
}

\bigskip

\newcommand{\f}{\frac}
\newcommand{\GS}{{\big(}\frac{\Gamma s}{(1+\Gamma s)^m}f(x){\big)}}

\def\p{\partial}
\def\G{\Gamma}

\def\r{\rho(t)}
\def\it{i\theta(t)}
\def\O{\Omega}
\def\t{\theta}
\def\v{\varphi}
\def\x{\xi}
\def\rb{\bar{\rho}}
\def\rbc{\f{\rho(1)^4}{1+\f{3\lambda^2\Gamma}{\Omega}\rho(1)^4(t'-1)}}
\def\fc{1+\f{3\lambda^2\Gamma}{\Omega}\rho(0)^4 t'}
\def\e{\epsilon}
\def\d{\delta}
\def\nab{\nabla}
\def\l{\bigg(}
\def\rr{\bigg)}
\def\wb{\underline{w}}

\begin{abstract}
In this paper, we explore the metastable states of nonlinear Klein-Gordon equations with potentials. 
These states come from instability of a bound state under a nonlinear \textit{Fermi's golden rule}. In \cite{SW}, Soffer and Weinstein studied the instability mechanism and obtained an anomalously slow-decaying rate $1/(1+t)^{\f14}$. Here we develop a new method to study the evolution of $L^2_x$ norm of solutions to Klein-Gordon equations. With this method, we prove a $H^1$ scattering result for Klein-Gordon equations with metastable states. By exploring the oscillations, with a dynamical system approach we also find a more robust and more intuitive way to derive the sharp decay rate $1/(1+t)^{\f14}$. 
\end{abstract}

\section{Introduction and Statement of Main Results}

In quantum mechanics, people observed some long-lived states, which last at least $10^2$ to $10^3$ times longer than the expectation. These long-lived states are called metastable states in {\color{black} the} physics literature.  
Mathematically, one would expect that these states carry anomalously slow-decaying rates. 

\subsection{Physical Motivation}
One way to produce a metastable state is through the instability of an excited state. Under tiny perturbations, energy of this (unstable) excited state starts to shift to ground state, free waves and nearby excited states. In this process, people observed anomalously long-lived states (metastable states).

To study the instability mechanism, with perturbation theory in 1927 Dirac did calculations in the following setting: Give two Hamiltonians $H_0$ and $H_1$ close to each each, assume they have eigenfunction (initial eigenstate) $i(x)$ and eigenfunction (final eigenstate) $f(x)$ respectively, Dirac calculated the transition probability per unit time from the state $i(x)$ to the state $f(x)$:
\begin{equation}\label{Dirac}
\Gamma_{i\rightarrow f}=\f{4\pi^2}{h}|\int_{\mathbb{R}^3}i(x)H_1(x)f(x) dx|^2\cdot \rho_f.
\end{equation}
Here $h$ is the Planck constant ($\approx 6.626\times 10^{-34}\quad kg\cdot m^2 \cdot s^{-1}$) and $\rho_f$ is the density of final states. 

In 1934 Fermi utilized (\ref{Dirac}) to establish his famous theory of  beta decay. In nuclear physics, beta decay is a type of radioactive decay in which a  $\beta-$ray (fast energetic electron or positron) and a neutrino are emitted from an atomic nucleus. In his paper, Fermi called (\ref{Dirac}) \textit{golden rule}. Later, in physics community, (\ref{Dirac}) is called \textit{Fermi's golden rule}.  
  
\subsection{Mathematical Motivation}
The mathematical study {\color{black} of} metastable states and Fermi's golden rule came quite late and remain largely open. 

The first complete analysis of such a problem for nonlinear PDE, was given in \cite{SW} by Soffer and Weinstein. Together with their further results in \cite{Sof-Wei1, Sof-Wei2}a nonlinear instability mechanism for a bound state of the Klein-Gordon equation with potential and an excited state of the Schr\"odinger equation with potential in $3+1$ dimensions are obtained. These {\color{black} instabilities lead to} metastable states. 

In \cite{SW} an anomalously  slow-decaying rate upper bound $1/(1+t)^{\f14}$ for metastable states was proved. {\color{black} More precisely},  Soffer and Weinstein \cite{SW} studied the following nonlinear Klein-Gordon equations (NLKG) in $3+1$ dimensions:
\begin{equation}\label{1.1}
\partial^2_{t}u-\Delta u+V(x) u+m^2 u=\lambda u^3, \quad \lambda \in \mathbb{R}\backslash 0.
\end{equation}
\begin{equation}\label{1.2}
u(0,x)=u_0(x), \quad \partial_{t}u(0,x)=u_1(x).
\end{equation}
They proved
\begin{theorem}\label{thm1.1}
Let $V(x)$ be real-valued and such that
\begin{itemize}
\item (V1) for $\d>5$ and $|\alpha|\leq 2$, $|\partial^{\alpha}V(x)|\leq C_{\alpha}(1+|x|^2)^{-\f{\d}{2}}$,
\item (V2) $(-\Delta+1)^{-1}\big( (x \cdot \nabla)^{l} V(x) \big) (-\Delta+1)^{-1}$ is bounded on $L^2$\\ for $|l|\leq N_{*}$ with $N_{*}\geq 10$. 
\item(V3) zero is not a resonance of the operator $-\Delta+V$. \footnote{We say that $0$ is a resonance for operator $L$ if there exists a solution $v$ of $L v=0$ such that $(1+|x|)^{-\frac{\gamma}{2}}v(x) \in L^2(\mathbb{R}^3)$ for any $\gamma>\f12$ but not for $\gamma=0$. Recall that a resonance at $0$ is an obstruction to prove dispersion. Our condition $(V3)$ removes this obstruction.}
\end{itemize}
Assume the operator
\begin{equation*}
B^2=-\Delta+V(x)+m^2
\end{equation*}
has continuous spectrum, $\sigma_{cont}(B^2)=[m^2, +\infty)$, and a unique strictly positive simple eigenvalue, $\Omega^2<m^2$ with associated normalized eigenfunction $\psi(x)$: 
\begin{equation*}
B^2\psi(x)=\O^2 \psi(x).
\end{equation*}
Assume the resonance condition (\textit{Fermi's Golden Rule Condition})
\begin{equation}\label{1.5}
\Gamma=\frac{\pi}{3\O}\int_{\mathbb{R}^3}P_c\psi^3(x) \d(B-3\O)P_c \psi^3(x) dx\equiv \f{\pi}{3\O}|(\mathcal{F}_c \psi^3)(3\O)|^2>0.
\end{equation}
Here, $P_c$ denotes the projection onto the continuous spectral part of $B$ and $\mathcal{F}_c$ denotes the Fourier transform relative to the continuous spectral part \footnote{$\mathcal{F}_c$ could be defined as a distorted Fourier transform. For $B^2=-\Delta+V(x)+m^2$, let $\phi(x,\lambda)$ be the unique functions satisfying
$$-\frac{d^2}{d\lambda^2}\phi(x,\lambda)+V(x)\phi(x,\lambda)+m^2\phi(x,\lambda)=\lambda^2 \phi(x,\lambda), \quad \phi(0,\lambda)=-1, \, \frac{d}{d\lambda}\phi(0,\lambda)=0.$$
For $\zeta>m$ we define
$$\mathcal{F}_c f(\zeta):=\int_{m}^{\infty}\phi(x,\zeta)f(x)dx. $$
} of $B$.
Assume that the initial data $u_0, u_1$ are such that the {\color{black}norms} $\|u_0\|_{W^{2,2}\cap W^{2,1}}$ and $\|u_1\|_{W^{1,2}\cap W^{1,1}}$ are sufficiently small.  

Then, the solution of the initial value problem for (\ref{1.1}) decays as $t\rightarrow +\infty$. 
\begin{equation*}
u(t,x)=R(t) \cos[\O t+\bar{\theta}(t)] \psi(x)+\eta(t,x), \quad \mbox{where}
\end{equation*}
$$R(t)=\mathcal{O}(|t|^{-\f14}), \quad \bar{\theta}(t)=\mathcal{O}(|t|^{\f12}), \quad \|\eta(t,x)\|_{L^8_{x}(\mathbb{R}^3)}\leq \f{1}{(1+t)^{\f34}}.$$
More precisely,
$$R(t)=\tilde{R}(t)+\mathcal{O}(|\tilde{R}(t)|^2), \quad (|\tilde{R}(t)| \, \, \mbox{small})$$
$$\mbox{and let}\, \tilde{R}(0)=\tilde{R}_0, \, \mbox{we have}\, \tilde{R}(t)=\tilde{R}_0(1+\f34 \tilde{R}^4_0 \O^{-1} \lambda^2 \Gamma |t|)^{-\f14}\cdot [1+\mathcal{O}(|t|^{-\delta})], \quad \quad \delta>0$$
$$R(0)=R_0, \quad R^2_0=|\int_{\mathbb{R}^3}\psi(x)u_0(x)dx|^2+\O^{-2}|\int_{\mathbb{R}^3}\psi(x)u_1(x)dx|^2.$$
\end{theorem}

The mathematical study of Fermi's golden rule is initiated from \cite{S93}. For Klein-Gordon equations, in \cite{SW} the energy transferring from one bound state to continuous spectrum is studied in details. In \cite{BC11}, a treatment for multi bound states is obtained. For Schr\"odinger's equation, Fermi's golden rule is explored in \cite{CM08, G15, GS07, GW08, GW11, Sof-Wei2, Ts, TY}. For further references on related problems we refer to \cite{B-Pe, B-S, FG14, FGS10, KMM, L-O-S, Miz08, Sof-Wei1, Sof-Wei3} and a survey paper \cite{Ko}.

\subsection{Main Results}
In this paper, by exploring the oscillations, we give a more intuitive way to derive asymptotic behaviors.  We also present a new method to obtain $H^1$ scattering for nonlinear Klein-Gordon equations with metastable states. The main theorems of this paper are
\begin{theorem}\label{thm1.2}
Under the same assumption of Theorem \ref{thm1.1} and an additional assumption $\int_{\mathbb{R}^3}u_0(x)\psi(x)dx\neq 0$, then solution $u(x,t)$ to (\ref{1.1}) has the following expansion as $t\rightarrow +\infty$:
\begin{equation}\label{solu}
u(t,x)=2\rho(t)\cos \theta(t) \psi(x)+\eta(t,x),
\end{equation}
where
\begin{equation}\label{eqn1.9}
\f{\f12\rho(0)}{\bigg( 1+\f{3\lambda^2\Gamma}{\O}\rho(0)^4 t \bigg)^{\f14} } \leq \rho(t) \leq \f{\f32\rho(0)}{\bigg( 1+\f{3\lambda^2\Gamma}{\O}\rho(0)^4 t \bigg)^{\f14} }.
\end{equation}

\begin{equation*}
\theta(t)-\O t=\mathcal{O}(t^{\f12}), \quad \mbox{and} \quad \|\eta(t,x)\|_{L^8_{x}(\mathbb{R}^3)}\leq \f{1}{(1+t)^{\f34}}.
\end{equation*}
\end{theorem}

\begin{remark}In \cite{SW}, Soffer and Weinstein employed normal form transformations. These transformations work for small data regime but not for large data. That's also the reason that in Theorem 1.1 they do not present an explicit expression for $R(t)$ depending on initial data. They have
$$R(t)=\tilde{R}(t)+\mathcal{O}(|\tilde{R}(t)|^2) \, \, \mbox{for} \, \, |\tilde{R}(t)| \, \, \mbox{small}.$$ 
In this paper, to derive Theorem 1.2 we use a different approach by exploring the oscillations. More details are presented in Sections 3-5. Here we give a more explicit expansion and avoid normal form transformations. This new approach could be used to derive some results in large data regime. 
\end{remark}

\begin{remark}
The constants $1/2$ and $3/2$ in (\ref{eqn1.9}) could be improved to $1-\tilde{\delta}$ and $1+\tilde{\delta}$ respectively, where $\tilde{\delta}$ is a fixed small positive number depending on initial data. See also Remark 15 in Section 5. 
\end{remark}

\begin{theorem}\label{thm1.3}
Under the same assumption of Theorem \ref{thm1.1}, there exist $S_1(x) \in L^2_x(\mathbb{R}^3)$ and $S_2(x) \in H^1_x(\mathbb{R}^3)$, such that as $t\rightarrow +\infty$

\begin{equation}\label{eqn1.10}
\|\eta(t,x)-\f{\sin Bt}{B} S_1(x)-\cos Bt S_2(x)\|_{H^1_x(\mathbb{R}^3)}\rightarrow 0.
\end{equation}
\end{theorem}
{\bf Note:} Since
\begin{equation*}
\|2\rho(t)\cos\theta(t)\psi(x)\|_{H^1_x(\mathbb{R}^3)}\rightarrow 0,
\end{equation*}
as $t\rightarrow +\infty$. Therefore, we also prove
\begin{theorem}(Forward $H^1$ Scattering)\label{thm1.4}
Under the same assumption of Theorem \ref{thm1.1}, for solution $u(t,x)$ in (\ref{solu}), there exist $S_1(x)$ and $S_2(x)$ \footnote{The detailed expressions of $S_1(x), S_2(x)$ are given in Theorem \ref{scattering states} in Section 6.} with finite $L^2_x(\mathbb{R}^3)$ and $H^1_x(\mathbb{R}^3)$ norms respectively, such that as $t\rightarrow +\infty$
\begin{equation*}
\|u(t,x)-\f{\sin Bt}{B} S_1(x)-\cos Bt S_2(x)\|_{H^1_x(\mathbb{R}^3)}\rightarrow 0.
\end{equation*}
\end{theorem}
\begin{remark}
Since $u(t,x)=2\rho(t)\cos \theta(t)\psi(x)+\eta(t,x)$, from $\rho(t)\approx 1/(1+t)^{\f14}$ and $\|\eta(t,x)\|_{L^8_x}\leq 1/(1+t)^{\f34}$, we only have $\|u(t,x)\|_{L^8_x}\leq 1/(1+t)^{\f14}$. Thus, It is not straight forward that 
$$\|u^3(t,x)\|_{L^1_tL^2_x(\mathbb{R},\mathbb{R}^3)}<+\infty.$$
Hence we cannot use the standard tools to prove scattering. In this paper, we develop a new method in Section 6 to prove Theorem \ref{thm1.4}. This new method is designed to study solutions' $L^2_x(\mathbb{R}^3)$ norms and to deal with the very slow decay rates of metastable states. 
\end{remark}

{\color{black}
\begin{remark}
For scattering, in 2011 Bambusi and Cuccagna \cite{BC11} obtained an interesting $H^1$ scattering result for Klein-Gordon equations with potentials in a slightly different setting (with an arbitrary number of possibly degenerate bound states) via normal form transformations and dispersive (Strichartz) estimates. Our scattering result here is consistent with their result when considering one bound state. The approach we employ here is quite different from theirs and we only use $L^2$ based estimates. Our method could provide some detailed estimates for decompositions of solutions. It would be very interesting to see how to generalize our approach to scenarios with multiple bound states. 

\end{remark}
}

\begin{remark}
For simplicity, in the following parts of this paper, for $p\geq 1$ we use $L^p_x$ to denote $L^p_x(\mathbb{R}^3)$. 
\end{remark}

\subsection{On Soliton Resolution Conjecture}
It is of great interest to study the large time behavior of nonlinear dispersive equations. This quest is aided by 
the \textit{soliton resolution conjecture}: asymptotic states of nonlinear dispersive equations should in general consist of free waves (solutions to the linear homogenous equations) and stable solitons moving apart from each other. 

However, for the generic physical circumstances there are also unstable coherent states. After their breakdown, metastable states may emerge, which could contribute to the final states. In the small data region as in \cite{SW} and this paper, Theorem 1.1 and Theorem 1.2 indicate a very slow decay rates for metastable states. One natural question is whether these metastable states could effect the conclusion of soliton resolution conjecture. {\color{black}In other words}, is it still true that solution $u(t,x)$ to (\ref{1.1}) could be decomposed into free waves and remainder terms, the $H^1_x$ norm of remainder terms go to $0$ as $t\rightarrow +\infty$?  

This is not a trivial question. Due to the slow decaying rate of $u(t,x)$, the standard scattering results do not follow.  In Section 6, we present a new method to study this question. And we obtain Theorems 1.3-1.5, which proves $H^1$ scattering and verifies the \textit{soliton resolution conjecture} in this setting. This means that even though we have constructed metastable states with both upper and lower bounds $1/(1+t)^{\f14}$ for a slow-decaying {\color{black} rate} to (\ref{1.1}), the \textit{soliton resolution conjecture} is as strong as it is before. 

For more complicate cases, results as in Theorems 1.2-1.5 remain open. 

\subsection{Main Difficulties}
To prove Theorems \ref{thm1.2}-\ref{thm1.4}, we would expect some difficulties. It is natural to ask
\begin{itemize}
\item Question 1: How to show the sharp decay rate for $\rho(t)$ as in (\ref{eqn1.9})? 
\item Question 2: We have  $\|\eta(t,x)\|_{L^8_{x}}\leq 1/(1+t)^{\f34}$ from \cite{SW}. But how to explore the detailed structures for $\eta(t,x)$ in $L^2_x$ as in (\ref{eqn1.10})? \\
\end{itemize}

\subsubsection{Question 1}
To study $\rho(t)$, we use polar coordinates and introduce a new ODE approach in Section 3. Under the ansatz (\ref{solu}), $\rho(t)$ would satisfy an ODE: 
\begin{equation}\label{rho with lot}
\begin{split}
&\rho'(t)+\frac{3\lambda^2}{4\Omega}\Gamma\r^5+\sum_{k\geq 1} c_k\|\psi\|^4_{L^4_x} \rho(t)^5 \sin k\theta(t)+\sum_{k\geq 1}d_k \|\psi\|^4_{L^4_x}\rho(t)^5 \cos k \theta(t)\\
&+\frac{\lambda}{\Omega}\|\psi\|^4_{L^4_x}\r^3 \sin{2\theta(t)}+\frac{\lambda}{2\Omega}\|\psi\|^4_{L^4_x}\r^3 \sin{4\theta(t)}+l.o.t.=0, 
\end{split}
\end{equation}
where $k$ stands for some positive integer; $c_k$ and $d_k$ are real numbers. Note that we have $\Gamma>0$ due to Fermi's golden rule condition (\ref{1.5}).
\begin{remark}
In this paper, we use $l.o.t.$ short for lower order terms.  And $l.o.t.$ are used to referring remainders, which decay faster than the leading term.  
\end{remark}
Equation (\ref{rho with lot}) implies
\begin{equation}\label{eqn1.13}
\begin{split}
\f{1}{\rho(t)^4}=&\f{(1+\f{3\lambda^2\Gamma}{\Omega}\rho(0)^4 t)}{\rho(0)^4}-\sum_{k}\int_0^t 4c_k \|\psi\|^4_{L^4_x} \sin k\theta(t') dt' \\
&-\sum_{k}\int_0^t 4d_k \|\psi\|^4_{L^4_x} \cos k\theta(t') dt' -\int_0^t \frac{4\lambda}{\Omega}\|\psi\|^4_{L^4_x} \f{\sin{2\theta(t')}}{\rho(t')^2}dt'\\
&-\int_0^t \frac{2\lambda}{\Omega}\|\psi\|^4_{L^4_x} \f{\sin{4\theta(t')}}{\rho(t')^2}dt'+l.o.t.\\
\end{split}
\end{equation}
Here $\sum_{k}$ is a finite sum of some integers $k$.

In Theorem \ref{thm1.2}, we hope to prove $1/\rho(t)^4\approx 1+t$ for $t$ large. This would give both lower and upper bounds for $\rho(t)$. 
Hence we only need to show that, on the right hand side of (\ref{eqn1.13}), the first term $(1+\f{3\lambda^2\Gamma}{\Omega}\rho(0)^4 t)/\rho(0)^4$ dominates. 

Luckily, for $\theta(t)$ we have $\theta(t)=\Omega t+l.o.t.$ When $t$ is large, $\sin k \theta(t)$ and $\cos k \theta(t)$ oscillate. Taking advantage of this observation, we can show that the second and third terms $\int_0^t 4c_k \|\psi\|^4_{L^4_x} \sin k\theta(t') dt' $ and $\int_0^t 4d_k \|\psi\|^4_{L^4_x} \cos k\theta(t') dt'$ are like constants. They are much smaller than the first term $(1+\f{3\lambda^2\Gamma}{\Omega}\rho(0)^4 t)/\rho(0)^4$. 

To deal with $\int_0^t \frac{4\lambda}{\Omega}\|\psi\|^4_{L^4_x} {\sin{2\theta(t')}}/{\rho(t')^2}dt'$ and $\int_0^t \frac{2\lambda}{\Omega}\|\psi\|^4_{L^4_x} {\sin{4\theta(t')}}/{\rho(t')^2}dt'$, we could use integration by {\color{black}parts} and employ the same method as for the second and third terms. However, we encounter an additional difficulty: since in \cite{SW} a lower bound for $\rho(t)$ has not been proved, we cannot rule out the possibility that $\rho(t)$ decays faster than $1/(1+t)^{\f14}$. And this would make $1/\rho(t')^2$ out of control. 

To overcome this difficulty, we construct a parametrix $\rb(t)\geq 0$ in Section 4 through
\begin{equation*}
\rb(t)^4=\f{\rho(0)^4}{1+\f{3\lambda^2\Gamma}{\Omega}\rho(0)^4 t} .
\end{equation*}
One can check $\rb(t)$ satisfies:
\begin{equation*}
\begin{split}
\rb'(t)&=-\f{3\lambda^2}{4\Omega}\Gamma\rb(t)^5,\\
\rb(0)&=\rho(0).
\end{split}
\end{equation*}
Then we introduce the unknown $\epsilon(t)$ through
\begin{equation*}
\r=\rb(t)(1+\epsilon(t)).
\end{equation*}
For initial data, we have $\e(0)=0$. Therefore, the seeking of {\color{black} a} lower bound for $\rho(t)$ is reduced to {\color{black}closing} a bootstrap argument for $\e(t)$ and to show that $\e(t)$ is small for all the time. And this is proved in Section 5. \\

\subsubsection{Question 2}
After deriving the asymptotic behavior for $\rho(t)$, we move to analyze the radiation field $\eta(t,x)$ in $L^2_x$ norm. 

For $\eta$, we have
\begin{equation*}
\eta(t,x)=\eta_1(t,x)+\eta_2(t,x)+\eta_3(t,x),
\end{equation*}
where 
\begin{equation*}
(\partial^2_t+B^2)\eta_1=0, \quad \eta_1(0,x)=P_c u_0, \quad \partial_t \eta_1(0,x)=P_c u_1,
\end{equation*}

\begin{equation}\label{eqn1.17}
(\partial^2_t+B^2)\eta_2=\lambda a^3 P_c \psi^3, \quad \eta_2(0,x)=0, \quad \partial_t \eta_2(0,x)=0,
\end{equation}

\begin{equation*}
(\partial^2_t+B^2)\eta_3=\lambda P_c(3a^2\psi^2\eta+3a\psi\eta^2+\eta^3), \quad \eta_3(0,x)=0, \quad \partial_t \eta_3(0,x)=0.
\end{equation*}
Here $a=a(t):=2\rho(t)\cos\theta(t)$.  Let's first focus on $\eta_2(t,x)$. From (\ref{eqn1.17}), we have

\begin{equation*}
\eta_2(t,x)=\lambda\int_0^t \f{\sin B(t-s)}{B} a^3(s)P_c \psi^3(x) ds. \footnote{Note that $B^2=-\Delta+V(x)+m^2$ has continuous spectrum $\sigma_{cont}(B^2)=[m^2, +\infty)$ and a unique strictly positive single eigenvalue $\Omega^2<m^2$. 
This implies that $1/B$ is a bounded operator for $L^2_x$. }
\end{equation*}
If we only use $|a(t)|^3\leq \rho(t)^3 \approx 1/(1+t)^{\f34}$ and standard dispersive estimates for wave operator (see Theorem 2.1 in \cite{SW}), we cannot even prove that $\eta_2(t,x)\in L^2_x$ for all $t\geq 0$. 

We overcome this difficulty by constructing an auxiliary function $w(t,x)$ through solving
\begin{equation*}
\begin{split}
(i\partial_t +B)w=&\lambda a^3 P_c \psi^3,\\
w(0,x)&=0.
\end{split}
\end{equation*}
Hence
\begin{equation}\label{def of w(t,x)}
w(t,x)=-i\int_0^t e^{iB(t-s)}\lambda a^3(s)P_c \psi^3(x) ds.
\end{equation}

\begin{remark}
At this moment, we don't know whether $w(t,x)\in L^2_x$. But later we will show that $w(t,x) \in L^2_x$. And this is an important step to prove $\eta_2(t,x)\in L^2_x$. 
\end{remark}

Recall $\psi(x)$ is a real valued function and $\|\psi\|_{L^2_x}=1$. From (\ref{def of w(t,x)}), we then have
\begin{equation*}
\mbox{Im}\,w(t,x)=-\int_0^t \sin B(t-s)\,\lambda a^3(s)P_c \psi^3(x) ds,
\end{equation*}
and
\begin{equation*}
\mbox{Re}\,w(t,x)=\int_0^t \cos B(t-s)\,\lambda a^3(s)P_c \psi^3(x) ds.
\end{equation*}
Define
\begin{equation*}
l(t):=\|w(t,x)\|^2_{L^2_x}=\|\mbox{Re}\,w(t,x)\|^2_{L^2_x}+\|\mbox{Im}\,w(t,x)\|^2_{L^2_x}.
\end{equation*}
Since $B$ is a self-adjoint operator, a simple calculation (see details in Section 6) implies 
\begin{equation*}
\begin{split}
\f{d}{dt}l(t)=-2\int_{\mathbb R^3} \mbox{Im}\,w(t,x)\lambda a^3(t)P_c\psi^3(x)dx.
\end{split}
\end{equation*}
Using the definitions of $\mbox{Re}\, u(t,x), \, \mbox{Im}\, u(t,x)$, together with the dispersive estimates for $B$ (see Theorem 2.1 in 
\cite{SW}) and the fact $|a(t)|\leq 1/(1+t)^{-\f14}$, we derive
\begin{equation*}
\begin{split}
\|\mbox{Re}\, w(t,x), \, \mbox{Im}\, w(t,x)\|_{L^8_x}\leq& \f{1}{(1+t)^{\f34}}.
\end{split}
\end{equation*}
With the estimate above, we arrive at
\begin{equation*}
\begin{split} 
l(t)\leq& l(0)+\int_0^t |\f{d}{dt}l(t)|dt\\
\leq&\int_0^t \|\mbox{Im}\, w(t,x)\|_{L^8_x} |a^3(t)| \|P_c\psi^3(x)\|_{L^{\f87}_x} dt\\
\leq& \int_0^t \f{1}{(1+t)^{\f34}}\cdot\f{1}{(1+t)^{\f34}}<+\infty.
\end{split}
\end{equation*}
Therefore, we have proved 
\begin{equation*}
\mbox{Im}\,w(t,x)=\int_0^t \sin B(t-s)\,\lambda a^3(s)P_c \psi^3(x) ds \in L^2_x,
\end{equation*}

\begin{equation*}
\mbox{Re}\,w(t,x)=\int_0^t \cos B(t-s)\,\lambda a^3(s)P_c \psi^3(x) ds \in L^2_x,
\end{equation*}
for any $t>0$. Since $1/B$ is a bounded operator for $L^2_x$, we deduce
\begin{equation*}
\eta_2(t,x)=\lambda\int_0^t \f{\sin B(t-s)}{B} a^3(s)P_c \psi^3(x) ds \in L^2_x. 
\end{equation*}

In order to prove Theorem \ref{thm1.3} and Theorem \ref{thm1.4}, {\color{black}we require a} more detailed analysis for $\eta_2(t,x)$ and $\eta_3(t,x)$ in $L^2_x$ norms. And we will show all the details in Section 6. 

\begin{remark}
The method to construct an auxiliary function $w(t,x)$ comes from considering Klein-Gordon equations as two coupled Dirac-type equations.  To estimate the $L^2_x$ norm of $\eta_2(t,x)$, we use the $L^2_x$ norm of $w(t,x)$ as a bridge. On one {\color{black}hand} $w(t,x)$ is the solution of one Dirac-type equation. It is easier to obtain more precise $L^2_x$ estimates. On the other {\color{black}hand} $w(t,x)$ carries the $L^2_x$ informations of $\eta_2(t,x)$. Based on $w(t,x)\in L^2_x$, we can prove that $\eta_2(t,x)\in L^2_x$.
\end{remark}

\section{Background}

\subsection{Setup}
For a small amplitude solution to (\ref{1.1}) and (\ref{1.2}), it is natural to decompose the solution as follows:

\begin{equation}\label{2.1}
u(x,t)=a(t)\psi(x)+\eta(t,x),
\end{equation}

\begin{equation*}
<\eta(t,x), \psi(x)>=0 \quad \mbox{for all} \quad t,
\end{equation*}
where $<f,g>$ denotes the usual complex inner product in $L^2_x$. 

Substitution of (\ref{2.1}) into (\ref{1.1}) gives
\begin{equation}\label{2.3}
a''(t)\psi(x)+\partial^2_t \eta(t,x)+\O^2 a(t)\psi(x)+B^2 \eta(t,x)=\lambda \big(a(t)\psi(x)+\eta(t,x)\big)^3.
\end{equation}
Taking the inner product of (\ref{2.3}) with $\psi$, we have 
\begin{equation} \label{2.4}
a''(t)+\O^2 a(t)=\lambda <\psi(x), \big(a(t)\psi(x)+\eta(t,x)\big)^3 >.
\end{equation}
Let $P_c$ denote the projection onto the continuous spectrum of $B^2$, i.e.
\begin{equation*}
P_c \nu\equiv \nu-<\psi, \nu>\psi.
\end{equation*}
Since $\eta=P_c\eta$, we then have
\begin{equation} \label{2.6}
\partial^2_t\eta+B^2\eta=\lambda P_c (a\psi+\eta)^3.
\end{equation}
Expand the cubic terms in (\ref{2.4}) and (\ref{2.6}):
\begin{equation} \label{2.7}
a''+\O^2 a=\lambda[a^3\int \psi^4+3a^2\int \psi^3\eta+3a\int \psi^2\eta^2+\int \psi \eta^3],
\end{equation}
\begin{equation*}
\partial^2_t \eta+B^2\eta=\lambda P_c (a^3\psi^3+3a^2\psi^2\eta+3a\psi\eta^2+\eta^3).
\end{equation*}
Recall that we have initial conditions
\begin{equation*}
a(0)=<\psi, u_0>, \quad a'(0)=<\psi, u_1>,
\end{equation*}

\begin{equation*}
\eta(0, x)=P_c u_0, \quad \partial_t \eta(0,x)=P_c u_1.
\end{equation*}

To proceed, we further decompose $\eta(t,x)$ into
\begin{equation}\label{eta}
\eta(t,x)=\eta_1(t,x)+\eta_2(t,x)+\eta_3(t,x),
\end{equation}
where for $\eta_1, \eta_2, \eta_3$ we have

\begin{equation}\label{eta1}
(\partial^2_t+B^2)\eta_1=0, \quad \eta_1(0,x)=P_c u_0, \quad \partial_t \eta_1(0,x)=P_c u_1,
\end{equation}

\begin{equation}\label{eta2}
(\partial^2_t+B^2)\eta_2=\lambda a^3 P_c \psi^3, \quad \eta_2(0,x)=0, \quad \partial_t \eta_2(0,x)=0,
\end{equation}

\begin{equation}\label{eta3}
(\partial^2_t+B^2)\eta_3=\lambda P_c(3a^2\psi^2\eta+3a\psi\eta^2+\eta^3), \quad \eta_3(0,x)=0, \quad \partial_t \eta_3(0,x)=0.
\end{equation}

\subsection{Key Resonant Damping Term}
We expect the function $a(t)$ to consist of oscillations coming from the natural frequency $\O$. We decompose $a(t)$ as

\begin{equation*}
a(t)=A(t)e^{i\O t}+\bar{A}(t)e^{-i\O t}.
\end{equation*}
Here $A(t)$ is a complex-valued function. And $\bar{A}(t)$ is its complex conjugate. 

To turn a second order ODE (\ref{2.7}) for $a(t)$ into a first order ODE for $A(t)$, we follow \cite{SW} and impose a gauge condition to fix $A(t)$:
\begin{equation}\label{gauge}
A'(t)e^{i\O t}+\bar{A}'(t)e^{-i\O t}=0.
\end{equation}
Together with (\ref{gauge}), we have
\begin{equation*}
\begin{split}
&a''(t)+\O^2 a(t)\\
=&[A(t)e^{i\O t}+\bar{A}(t)e^{-i\O t}]''+\O^2 [A(t)e^{i\O t}+\bar{A}(t)e^{-i\O t}]\\
=&[A'(t)e^{i\O t}+\bar{A}'(t)e^{-i\O t}]'(t)+[i\O A(t)e^{i\O t}-i\O \bar{A}(t)e^{-i\O t}]'\\
&+\O^2[A(t)e^{i\O t}+\bar{A}(t)e^{-i\O t}]\\
=&[i\O A'(t)e^{i\O t}-\O^2 A(t)e^{i\O t}-i\O \bar{A}'(t)e^{-i\O t}- \O^2 \bar{A}(t)e^{-i\O t}]\\
&+\O^2[A(t)e^{i\O t}+\bar{A}(t)e^{-i\O t}]\\
=&i\O A'(t)e^{i\O t}-i\O \bar{A}'(t)e^{-i\O t}\\
=&2i \O A'(t) e^{i\O t}.
\end{split}
\end{equation*}
Equation (\ref{2.7}) then is reduced to the first order equation:
\begin{equation}\label{2.17}
A'(t)=(2i\O)^{-1}e^{-i\O t} F(a,\eta),
\end{equation}
where $F(a,\eta)$ is the sum of the following terms
\begin{equation*}
\begin{split}
F_1(a)=&\lambda a^3\int \psi^4,\\
F_2(a, \eta_j)=&3\lambda a^2\int \psi^3 \eta_j, \quad j=1,2,3\\
F_3(a, \eta)=&3 \lambda a\int \psi^2 \eta^2,\\
F_4(\eta)=&\lambda \int \psi \eta^3.\\
\end{split}
\end{equation*}

There is a key resonant term arising from $F_2(a,\eta_2)$. We analyze this key term. And for the detailed expression of other terms, we cite \cite{SW}. From (\ref{eta2}) we have
\begin{equation*}
\begin{split}
\eta_2=&\lambda \int_0^t\frac{\sin B(t-s)}{B} a^3(s)ds P_c \psi^3.\\
\end{split}
\end{equation*}
We first break $\eta_2$ into two parts: $\eta^r_2$ and $\eta^{nr}_2$. Here $\eta^r_2$ contains the resonance and $\eta^{nr}_2$ is the non-resonant part.

\begin{proposition}
\begin{equation*}
\eta_2=\eta_2^r+\eta^{nr}_2,
\end{equation*}
where
\begin{equation*}
\eta^r_2=\f{\lambda}{2iB}e^{iBt}\int_0^t e^{-is(B-3\O)}A^3(s)ds P_c\psi^3,
\end{equation*}
and
\begin{equation*}
\begin{split}
\eta^{nr}_2=&\f{3\lambda}{2iB}e^{iBt}\int_0^t e^{-is(B-\O)}A^2(s)\bar{A}(s)ds P_c\psi^3\\
&+\f{3\lambda}{2iB}e^{iBt}\int_0^t e^{-is(B+\O)}\bar{A}^2(s)A(s)ds P_c\psi^3\\
&+\f{\lambda}{2iB}e^{iBt}\int_0^t e^{-is(B+3\O)}\bar{A}^3(s)ds P_c\psi^3\\
&-\f{\lambda}{2iB}e^{-iBt}\int_0^t e^{is(B+3\O)}A^3(s)ds P_c\psi^3\\
&-\f{3\lambda}{2iB}e^{-iBt}\int_0^t e^{is(B+\O)}A^2(s)\bar{A}(s)ds P_c\psi^3\\
&-\f{3\lambda}{2iB}e^{-iBt}\int_0^t e^{is(B-\O)}\bar{A}^2(s)A(s)ds P_c\psi^3\\
&-\f{\lambda}{2iB}e^{-iBt}\int_0^t e^{is(B-3\O)}\bar{A}^3(s)ds P_c\psi^3\\
\equiv& \sum_{j=1}^7 \eta_{2j}^{nr}.
\end{split}
\end{equation*}
\end{proposition}
We now focus on $\eta_2^r$. In order to study $\eta^r_2$ near the resonant point $3\O$ in the continuous spectrum of $B$, we introduce a regularization of $\eta^r_2$.
For $\epsilon>0$, let
\begin{equation*}
\eta^r_{2\epsilon}=\f{\lambda}{2iB}e^{iBt}\int_0^t e^{-is(B-3\O+i\epsilon)}A^3(s)ds P_c\psi^3.
\end{equation*}
Note that $\eta^{r}_2=\lim_{\epsilon\rightarrow 0} \eta^r_{2\epsilon}$ in the sense of a distribution acts on a function. \\

The following result, proved using integration by parts, isolates the key resonant term. Denote $A_0=A(0)$. We have

\begin{proposition}
For $\epsilon> 0$,
\begin{equation}\label{2.24}
\begin{split}
\eta^r_{2\epsilon}=&\f{\lambda}{2}[B(B-3\O+i\epsilon)]^{-1}e^{3i\O t} A^3(t) e^{\epsilon t} P_c\psi^3\\
&-\f{\lambda}{2}A_0^3[B(B-3\O+i\epsilon)]^{-1}e^{iB t} P_c\psi^3\\
&-\f{3\lambda}{2}[B(B-3\O+i\e)]^{-1}e^{iBt}\int_0^t e^{-is(B-3\O+i\e)}A^2 A' ds P_c \psi^3\\
=&\eta^r_{*\e}+\eta^{nr1}_{*\e}+\eta^{nr2}_{*\e}.
\end{split}
\end{equation}
\end{proposition}
Note $F_2(a, \eta_2)=3\lambda a^2\int_{\mathbb{R}^3} \psi^3(x) \eta_2(t,x)dx$. For $F_2(a, \eta_2)$ we have
\begin{equation*}
\begin{split}
F_2(a,\eta_2)=& F_2(a, \eta^r_2)+F_2(a, \eta_2^{nr})\\
=&\lim_{\e\to 0} F_2(a, \eta^r_{*\e})+\lim_{\e\to 0} F_2 (a, \eta^{nr1}_{*\e}+\eta^{nr2}_{*\e})+F_2(a, \eta^{nr}_2)\\
=&F_2(a, \eta^{r}_*)+F_2(a, \eta^{nr1}_{*}+\eta^{nr2}_{*})+F_2(a,\eta^{nr}_2).
\end{split}
\end{equation*}
What follows now is a detailed expansion of the term $\lim_{\e\to 0} F_2(a, \eta^r_{*\e})$. Let

\begin{equation*}
\begin{split}
\Lambda:=& \lim_{\e\to 0} <P_c \psi^3, \f{1}{B} \f{B-3\O}{(B-3\O)^2+\e^2} P_c \psi^3>\\
=&<P_c \psi^3, \f{1}{B} P.V. \f{1}{B-3\O} P_c \psi^3>, \mbox{and}
\end{split}
\end{equation*}

\begin{equation*}
\begin{split}
\Gamma:=& \lim_{\e\to 0}<P_c \psi^3, \f{1}{B} \f{\e}{(B-3\O)^2+\e^2} P_c \psi^3>\\
=&\f{\pi}{3\O}<P_c \psi^3, \d(B-3\O) P_c \psi^3>\\
=& \f{\pi}{3\O}|\mathcal{F}_c[\psi^3](3\O)|^2.
\end{split}
\end{equation*}
By hypothesis (\ref{1.5}), it follows $\Gamma>0$. We now substitute the expression for $\eta^r_{*\e}$. Put  (\ref{2.24}) into the definition of $F_2(a, \eta^r_{*\e})$. Let $\e \to 0$ and use the distributional identity:

\begin{equation*}
(x\pm i0)^{-1}:= \lim_{\e\to 0} (x\pm i\e)^{-1}= P.V. x^{-1}\mp i\pi \d(x).
\end{equation*}
It yields

\begin{equation}\label{2.29}
\begin{split}
F_2(a, \eta^r_{*}):=&\lim_{\e \to 0} F_2(a, \eta^r_{*\e})\\
=&\f32 \lambda^2 (\Lambda-i\Gamma)[|A|^4A e^{i\O t}+A^5 e^{5i\O t}+2|A|^2A^3 e^{3i\O t}].
\end{split}
\end{equation}
Multiplying (\ref{2.29}) by $(2i\O)^{-1} e^{-i\O t}$, we see that
\begin{proposition}
\begin{equation*}
(2i\O)^{-1} e^{-i\O t} F_2(a, \eta^r_{*})=-\f34 \f{\lambda^2}{\O} (i\Lambda+\Gamma)[|A|^4A +A^5 e^{4i\O t}+2|A|^2A^3 e^{2i\O t}].
\end{equation*}
\end{proposition}
With (\ref{2.17}):
\begin{equation*}
A'(t)=(2i\O)^{-1}e^{-i\O t} F(a,\eta),
\end{equation*}
this proposition implies
\begin{equation}\label{2.31}
\begin{split}
A'(t)=&-\f34 \f{\lambda^2 \Gamma}{\O}|A|^4A-\f{3i}{4} \f{\lambda^2 \Lambda}{\O}|A|^4A-\f34 \f{\lambda^2}{\O} (i\Lambda+\Gamma)[A^5 e^{4i\O t}+2|A|^2A^3 e^{2i\O t}]\\
&+\mbox{remainder terms.}
\end{split}
\end{equation}
\begin{remark}Since $\Gamma, \Omega$ are positive, $-\f34 \f{\lambda^2 \Gamma}{\O}|A|^4A$ is the key resonant damping term for (\ref{2.31}).\end{remark}

\section{A New ODE Approach}
Define $$\underline{\rho}(\zeta):=<P_c \psi^3, B^{-1}(B-\zeta)^{-1}P_c\psi^3>.$$
The complete description of (\ref{2.31}) is as follows (see also section 4 of \cite{SW}): 

 \begin{proposition}
 The amplitude $A(t)$ satisfies the equation
 \begin{equation}\label{3.1}
\begin{split}
&A'(t)+\frac{3\lambda^2}{4\Omega}\Gamma|A|^4A\\
&+\frac{i\lambda}{2\Omega}\|\psi\|^4_4 3|A|^2A+\frac{3i\lambda^2}{4\Omega}[\Lambda-5\underline{\rho}(\Omega)+3\underline{\rho}(-\Omega)+\underline{\rho}(-3\Omega)]|A|^4A\\
=&\frac{-i\lambda}{2\Omega}\|\psi\|^4_4(A^3e^{2i\Omega t}+3|A|^2\bar{A}e^{-2i\Omega t}+\bar{A}^3e^{-4i\Omega t})\\
&-\f{3\lambda^2}{4\Omega}A^5 e^{4i\O t}[i\Lambda+\Gamma-\underline{\rho}(-3\Omega)]\\
&-\f{3i\lambda^2}{4\Omega}|A|^2A^3e^{2i\Omega t}[-2\Lambda+2i\Gamma+\underline{\rho}(\Omega)+i\underline{\rho}(-3\Omega)+3\underline{\rho}(-\Omega)]\\
&-\f{3i\lambda^2}{4\Omega}|A|^4\bar{A}e^{-2i\Omega t}[7\underline{\rho}(\Omega)+9\underline{\rho}(-\Omega)+\underline{\rho}(-3\Omega)+\underline{\rho}(3\Omega+i0)]\\
&-\f{3i\lambda^2}{4\Omega}|A|^2\bar{A}^3e^{-4i\Omega t}[3\underline{\rho}(-\Omega)-2i\underline{\rho}(-3\Omega)+\underline{\rho}(3\Omega+i0)+3\underline{\rho}(\Omega)]\\
&-\f{3i\lambda^2}{4\Omega}\bar{A}^5e^{-6i\Omega t}[i\underline{\rho}(-3\Omega)-\frac43 i\underline{\rho}(3\Omega+i0)]+E,
\end{split}
\end{equation}
 where
\begin{equation*}
\begin{split}
E=&(2i\O)^{-1}e^{-i\O t}3\lambda a\int \psi\eta^2+(2i\O)^{-1}\lambda e^{-i\O t}\int \psi \eta^3+\sum_{j=1}^7 E_{2j}^{nr}\\
&+(2i\O)^{-1}e^{-i\O t}[F_2(a,\eta_1)+F_2(a,\eta_3)]+(2i\O)^{-1}e^{-i\O t}F_2(a, \eta_{*}^{nr1}+\eta_{*}^{nr2}).
\end{split}
\end{equation*}
The detailed expression of $E^{nr}_{2j}$ could also be found on page 43 in \cite{SW}.

 \end{proposition}

 In \cite{SW}, Soffer and Weinstein constructed a smooth near-identity change of variables $A(t)\rightarrow \tilde{A}(t)$ with the properties
 \[
 \tilde{A}=A+h(A,t),
 \]

 \[
 h(A,t)=O(|A|^3), |A|\to 0,
 \]

 \[
 h(A,t)=h(A, t+2\pi \O^{-1}).
 \]
 And in terms of $\tilde{A}$, (\ref{3.1}) becomes
 \begin{equation*}
\begin{split}
 \tilde{A}'=i\lambda c_{21}|\tilde{A}|^2\tilde{A}+\lambda^2 d_{32}|\tilde{A}|^4\tilde{A}+i\lambda^2 c_{32} |\tilde{A}|^4 \tilde{A}+\mathcal{O}(|\tilde{A}|^7)+l.o.t.,
 \end{split}
\end{equation*}
where $\lambda^2 d_{32}=-\f34 \frac{\lambda^2}{\O}\Gamma<0$ and the constants $c_{21}$ and $c_{32}$ are real numbers.
This implies
\begin{equation}\label{3.4}
\partial_t |\tilde{A}|^2=-\f{3}{8}\f{\lambda^2}{\O}\Gamma |\tilde{A}|^6+\mathcal{O}(|\tilde{A}|^8)+l.o.t..
\end{equation}
Together with ODE analysis based on (\ref{3.4}) and smallness assumption on $|\tilde{A}(0)|$ , Soffer and Weinstein proved an upper bound for $|\tilde{A}|(t)$ in \cite{SW}. This implies an upper bound for $|A|(t)$.

In this section, we introduce a new approach without using near-identity change of variables. The new approach is more explicit and will give both lower and upper bounds for $|A|(t)$. We multiply $e^{i\Omega t}$ on both sides of (\ref{3.1}) and arrive at
\begin{equation}\label{For X(t)}
\begin{split}
&e^{i\Omega t}A'(t)+\frac{3\lambda^2}{4\Omega}\Gamma|A|^4Ae^{i\Omega t}\\
&+\frac{i\lambda}{2\Omega}\|\psi\|^4_4 3|A|^2Ae^{i\Omega t}+\frac{3i\lambda^2}{4\Omega}[\Lambda-5\underline{\rho}(\Omega)+3\underline{\rho}(-\Omega)+\underline{\rho}(-3\Omega)]|A|^4Ae^{i\Omega t}\\
&+\frac{i\lambda}{2\Omega}\|\psi\|^4_4(A^3e^{3i\Omega t}+3|A|^2\bar{A}e^{-i\Omega t}+\bar{A}^3e^{-3i\Omega t})\\
=&-\f{3\lambda^2}{4\Omega}A^5e^{5i\Omega t}[i\Lambda+\Gamma-\underline{\rho}(-3\Omega)]\\
&-\f{3i\lambda^2}{4\Omega}|A|^2A^3e^{3i\Omega t}[-2\Lambda+2i\Gamma+\underline{\rho}(\Omega)+i\underline{\rho}(-3\Omega)+3\underline{\rho}(-\Omega)]\\
&-\f{3i\lambda^2}{4\Omega}|A|^4\bar{A}e^{-i\Omega t}[7\underline{\rho}(\Omega)+9\underline{\rho}(-\Omega)+\underline{\rho}(-3\Omega)+\underline{\rho}(3\Omega+i0)]\\
&-\f{3i\lambda^2}{4\Omega}|A|^2\bar{A}^3e^{-3i\Omega t}[3\underline{\rho}(-\Omega)-2i\underline{\rho}(-3\Omega)+\underline{\rho}(3\Omega+i0)+3\underline{\rho}(\Omega)]\\
&-\f{3i\lambda^2}{4\Omega}\bar{A}^5e^{-5i\Omega t}[i\underline{\rho}(-3\Omega)-\frac43 i\underline{\rho}(3\Omega+i0)]+e^{i\Omega t}E.
\end{split}
\end{equation}
In small data regime, with standard dispersive estimates (see for example Proposition 7.5 in \cite{SW}), we have control of $E$
\begin{equation}\label{estimate E}
|E(t)|\leq \d^2_0(1+t)^{-\f54-\delta} \quad \mbox{with} \quad \delta\geq 0,
\end{equation}
where $\d_0$ is a small constant depending on initial data $u_0$ and $u_1$. \footnote{Here we could use Proposition 7.5 in \cite{SW}. We can choose $\delta_0$ such that $\d^2_0\leq Q_0(A,\eta)$. In \cite{SW} $Q_0(A,\eta)$ is a small number depending on initial data $u_0$ and $u_1$.} And here $(1+t)^{-\f54-\delta}$ stands for a decay rate faster than $(1+t)^{-\f54}$.

\begin{remark} The right hand side of (\ref{For X(t)}) contains many oscillation phases $e^{ki\O t}$.  Our new observation is that if we let $X:=e^{i\Omega t}A(t)$, the right hand side of (\ref{For X(t)}) become 
\begin{equation*}
\begin{split}
&-\f{3\lambda^2}{4\Omega}X^5[i\Lambda+\Gamma-\underline{\rho}(-3\Omega)]\\
&-\f{3i\lambda^2}{4\Omega}|X|^2 X^3[-2\Lambda+2i\Gamma+\underline{\rho}(\Omega)+i\underline{\rho}(-3\Omega)+3\underline{\rho}(-\Omega)]\\
&-\f{3i\lambda^2}{4\Omega}|X|^4\bar{X}[7\underline{\rho}(\Omega)+9\underline{\rho}(-\Omega)+\underline{\rho}(-3\Omega)+\underline{\rho}(3\Omega+i0)]\\
&-\f{3i\lambda^2}{4\Omega}|X|^2\bar{X}^3[3\underline{\rho}(-\Omega)-2i\underline{\rho}(-3\Omega)+\underline{\rho}(3\Omega+i0)+3\underline{\rho}(\Omega)]\\
&-\f{3i\lambda^2}{4\Omega}\bar{X}^5[i\underline{\rho}(-3\Omega)-\frac43 i\underline{\rho}(3\Omega+i0)]+e^{i\Omega t}E.
\end{split}
\end{equation*}
The oscillation phases $e^{ki\O t}$ are absorbed. 
\end{remark}

 With the fact
$$e^{i\Omega t} A'(t)=(e^{i\Omega t}A(t))'-i\Omega e^{i\Omega t}A(t)=X'(t)-i\O X,$$
we rewrite (\ref{For X(t)})
\begin{equation}\label{X(t) old}
\begin{split}
&X'(t)+\frac{3\lambda^2}{4\Omega}\Gamma|X|^4X-i\Omega X\\
&+\frac{i\lambda}{2\Omega}\|\psi\|^4_4 3|X|^2X+\frac{3i\lambda^2}{4\Omega}[\Lambda-5\underline{\rho}(\Omega)+3\underline{\rho}(-\Omega)+\underline{\rho}(-3\Omega)]|X|^4X\\
&+\frac{i\lambda}{2\Omega}\|\psi\|^4_4(X^3+3|X|^2\bar{X}+\bar{X}^3)\\
=&-\f{3\lambda^2}{4\Omega}X^5[i\Lambda+\Gamma-\underline{\rho}(-3\Omega)]\\
&-\f{3i\lambda^2}{4\Omega}|X|^2 X^3[-2\Lambda+2i\Gamma+\underline{\rho}(\Omega)+i\underline{\rho}(-3\Omega)+3\underline{\rho}(-\Omega)]\\
&-\f{3i\lambda^2}{4\Omega}|X|^4\bar{X}[7\underline{\rho}(\Omega)+9\underline{\rho}(-\Omega)+\underline{\rho}(-3\Omega)+\underline{\rho}(3\Omega+i0)]\\
&-\f{3i\lambda^2}{4\Omega}|X|^2\bar{X}^3[3\underline{\rho}(-\Omega)-2i\underline{\rho}(-3\Omega)+\underline{\rho}(3\Omega+i0)+3\underline{\rho}(\Omega)]\\
&-\f{3i\lambda^2}{4\Omega}\bar{X}^5[i\underline{\rho}(-3\Omega)-\frac43 i\underline{\rho}(3\Omega+i0)]+e^{i\Omega t}E.
\end{split}
\end{equation}

\begin{remark} In this paper, we are studying a small data problem. Later we will see that the right hand side of (\ref{X(t) old}) are lower order terms. The precise coefficients are not important. For convenience, in the rest part of this paper we will write the key terms of an equation on its left hand side and lower order terms on its right.  We thus only keep the coefficients on the left precise. And the complex-valued coefficients depending on $\Lambda, \Gamma, \underline{\rho}, \Omega$ on the right are all treated as $1$.  In this fashion, we rewrite (\ref{X(t) old}) into
\begin{equation}\label{1}
\begin{split}
&X'(t)+\frac{3\lambda^2}{4\Omega}\Gamma|X|^4X-i\Omega X\\
&+\frac{i\lambda}{2\Omega}\|\psi\|^4_4 3|X|^2X+\frac{3i\lambda^2}{4\Omega}[\Lambda-5\underline{\rho}(\Omega)+3\underline{\rho}(-\Omega)+\underline{\rho}(-3\Omega)]|X|^4X\\
&+\frac{i\lambda}{2\Omega}\|\psi\|^4_4(X^3+3|X|^2\bar{X}+\bar{X}^3)\\
=&X^5+|X|^2X^3+|X|^4\bar{X}+|X|^2\bar{X}^3+\bar{X}^5+e^{i\O t} E.
\end{split}
\end{equation}
Note that $\Omega$ and $\Gamma$ are positive constants.
\end{remark}

\begin{remark}To explore the oscillation structures hidden in complex-valued function $X(t)$, we further rewrite $X(t)$ in polar coordinates: $X(t)=\r e^{\it}$, where $\rho(t), \theta(t)$ are real functions.
\end{remark}

From  (\ref{1}), we get
\begin{equation}\label{2}
\begin{split}
&\rho'(t)e^{\it}+i\r\theta'(t)e^{\it}+\frac{3\lambda^2}{4\Omega}\Gamma\r^5 e^{\it}-i\O\r e^{\it}\\
&+\frac{i\lambda}{2\Omega}\|\psi\|^4_4 3\r^3 e^{\it}+\frac{3i\lambda^2}{4\Omega}[\Lambda-5\underline{\rho}(\Omega)+3\underline{\rho}(-\Omega)+\underline{\rho}(-3\Omega)]\r^5 e^{\it}\\
&+\frac{i\lambda}{2\Omega}\|\psi\|^4_4\r^3 (e^{3i\theta(t)}+3e^{-i\theta(t)}+e^{-3i\theta(t)})\\
=&\r^5e^{5\it}+\r^5e^{3\it}+\r^5e^{-\it}+\r^5e^{-3\it}+\r^5e^{-5\it}+e^{i\O t}E(t).
\end{split}
\end{equation}
Here $\r$ is nonnegative. Multiply $e^{-\it}$ on both sides of (\ref{2}), we derive

\begin{equation}\label{3}
\begin{split}
&\rho'(t)+i\r\theta'(t)+\frac{3\lambda^2}{4\Omega}\Gamma\r^5-i\O\r\\
&+\frac{i\lambda}{2\Omega}\|\psi\|^4_4 3\r^3+\frac{3i\lambda^2}{4\Omega}[\Lambda-5\underline{\rho}(\Omega)+3\underline{\rho}(-\Omega)+\underline{\rho}(-3\Omega)]\r^5\\
&+\frac{i\lambda}{2\Omega}\|\psi\|^4_4\r^3 (e^{2i\theta(t)}+3e^{-2i\theta(t)}+e^{-4i\theta(t)})\\
=&\r^5e^{4\it}+\r^5e^{2\it}+\r^5e^{-2\it}+\r^5e^{-4\it}+\r^5e^{-6\it}\\
&+e^{-i\t(t)}e^{i\O t}E(t).
\end{split}
\end{equation}
Denote $M$ through
\begin{equation}\label{M}
\begin{split}
&M(e^{\it})\\
=&\r^5e^{4\it}+\r^5e^{2\it}+\r^5e^{-2\it}+\r^5e^{-4\it}+\r^5e^{-6\it}.
\end{split}
\end{equation}

We thus obtain
\begin{equation}\label{4}
\begin{split}
&\rho'(t)+i\r\theta'(t)+\frac{3\lambda^2}{4\Omega}\Gamma\r^5-i\O\r\\
&+\frac{i\lambda}{2\Omega}\|\psi\|^4_4 3\r^3+\frac{3i\lambda^2}{4\Omega}[\Lambda-5\underline{\rho}(\Omega)+3\underline{\rho}(-\Omega)+\underline{\rho}(-3\Omega)]\r^5\\
&+\frac{i\lambda}{2\Omega}\|\psi\|^4_4\r^3 (4\cos{2\theta(t)}+\cos{4\theta(t)})\\
&+\frac{\lambda}{2\Omega}\|\psi\|^4_4\r^3 (2\sin{2\theta(t)}+\sin{4\theta(t)})\\
=&M(e^{\it})+e^{-i\t(t)}e^{i\O t}E(t).
\end{split}
\end{equation}
{\bf Note:} Although $M(e^{i\theta(t)})$ contains the leading factor $\rho(t)^5$ as in the key term $3\lambda^2\Gamma \rho(t)^5/4\Omega$, later we will see that due to the phase factor $e^{ik\theta(t)}$ ($k\neq 0$), $M(e^{i\theta(t)})$ actually serves as a lower order term.

Take the real part of (\ref{3}), we have
\begin{equation}\label{rho}
\begin{split}
&\rho'(t)+\frac{3\lambda^2}{4\Omega}\Gamma\r^5+\frac{\lambda}{\Omega}\|\psi\|^4_4\r^3 \sin{2\theta(t)}+\frac{\lambda}{2\Omega}\|\psi\|^4_4\r^3 \sin{4\theta(t)}\\
=&Re(M(e^{\it}))+Re(e^{-i\t(t)}e^{i\O t}E(t)).
\end{split}
\end{equation}
Take the imaginary part of (\ref{4}), we have
\begin{equation}\label{theta0}
\begin{split}
&\theta'(t)\r-\Omega\r+\frac{3\lambda}{2\Omega}\|\psi\|^4_4\r^3\\
&+\frac{3\lambda^2}{4\Omega}[\Lambda-5\underline{\rho}(\Omega)+3\underline{\rho}(-\Omega)+\underline{\rho}(-3\Omega)]\r^5\\
&+\frac{2\lambda}{\Omega}\|\psi\|^4_4 \cos{2\theta(t)}\r^3+\frac{\lambda}{2\Omega}\|\psi\|^4_4 \cos{4\theta(t)}\r^3\\
=&Im(M(e^{\it}))+Im(e^{-i\t(t)}e^{i\O t}E(t)).
\end{split}
\end{equation}
\begin{remark}
From the definition of $M(e^{\it})$, we have 
$$|Re(M(e^{i\theta(t)}))|\leq \r^5, \quad \mbox{and} \quad |Im(M(e^{i\theta(t)}))|\leq \r^5.$$
\end{remark}

\begin{remark}
At first glance, changing of variables in Section 2 and Section 3 is a little bit complicated. But the logic behind is quite natural: 
By imposing a gauge condition (\ref{gauge}), we first change a second order ODE (\ref{2.7}) for $a(t)$ into a first order ODE (\ref{3.1}) for $A(t)$. We then want to find a systematical way to deal with the phase terms $e^{ik\O t}$ showing on the right of (\ref{3.1}). 
Observe that by multiplying $e^{i\O t}$ on both sides of (\ref{3.1}) and denoting $X(t):=e^{i\O t} A(t)$, we have a cleaner equation (\ref{X(t) old}) for $X(t)$. To further explore the oscillation structures of $X(t)$, we rewrite $X(t)$ in polar coordinates: $X(t)=\rho(t)e^{i\theta(t)}$. This implies an equation (\ref{4}) for $\theta(t)$ and $\theta(t)$. Taking the real and imaginary parts of (\ref{4}) respectively, we hence obtain (\ref{rho}) and (\ref{theta0}). 
\end{remark}

\section{Parametrix Construction}
We construct parametrix $\rb(t)$ through

\begin{equation}\label{rhob4}
\rb(t)^4=\f{\rho(0)^4}{1+\f{3\lambda^2\Gamma}{\Omega}\rho(0)^4 t} .
\end{equation}
Note that $\rb(t)\geq 0$ and $\rb(t)$ is the solution to the following ODE:

\begin{equation}\label{rhob}
\begin{split}
\rb'(t)&=-\f{3\lambda^2}{4\Omega}\Gamma\rb(t)^5,\\
\rb(0)&=\rho(0).
\end{split}
\end{equation}
We introduce the unknown $\epsilon(t)$ through
\begin{equation*}
\r=\rb(t)(1+\epsilon(t)).
\end{equation*}
Here we have $\e(0)=0$. From (\ref{rho}) we derive
\begin{equation*}
\begin{split}
&\rb'(t)(1+\e(t))+\rb(t)\e'(t)\\
&+\f{3\lambda^2\Gamma}{4\Omega}\rb(t)^5(1+\e(t))^5+\f{\lambda}{\Omega}\|\psi\|^4_4\rb^3(t)(1+\e(t))^3\sin2\theta(t)\\
&+\f{\lambda}{2\Omega}\|\psi\|^4_4\rb^3(t)(1+\e(t))^3\sin4\theta(t)\\
=&Re(M(e^{i\theta(t)}))+Re(e^{-i\theta(t)}e^{i\Omega t}E(t)).
\end{split}
\end{equation*}
With the fact $\rb'(t)=-\f{3\lambda^2}{4\Omega}\Gamma\rb(t)^5$, we arrive at the equation for $\e(t)$:

\begin{equation}\label{epsilon}
\begin{split}
&\e'(t)+\f{3\lambda^2\Gamma}{\Omega}\rb(t)^4\e(t)\\
=&\rb(t)^4\e(t)^2+\rb(t)^4\e(t)^3+\rb(t)^4\e(t)^4+\rb(t)^4\e(t)^5\\
&+\rb(t)^2(1+\e(t))^3\sin2\theta(t)+\rb(t)^2(1+\e(t))^3\sin4\theta(t)\\
&+\f{1}{\rb(t)}Re(M(e^{i\theta(t)}))+\f{1}{\rb(t)}Re(e^{-i\theta(t)}e^{i\Omega t}E(t)).
\end{split}
\end{equation}

\section{Bootstrap Argument}

In this section, we will employ a method of bootstrap to derive a sharp decay rate for $\rho(t)$. We set the bootstrap assumption
\begin{equation}\label{bootstrap}
|\epsilon(t)|\leq 3\delta^{\f12}_0.
\end{equation}
In the remaining part of this subsection, we will prove a sharper bound than (\ref{bootstrap}) and show that 
\begin{equation}\label{bootstrap2}
|\epsilon(t)|\leq 2\delta^{\f12}_0.
\end{equation}
\begin{remark}
Bootstrap argument is a standard tool now to derive estimates for nonlinear ODEs and PDEs. In this paper, $\epsilon(t)$ satisfies an ODE (\ref{epsilon}) with initial data $\epsilon(0)=0$.  We want to show that the interval $I$ for time $t$ satisfying (\ref{bootstrap}) is both close and open. By standard ODE theory, $I$ is {\color{black}closed}. Under the assumption of (\ref{bootstrap}), if (\ref{bootstrap2}) could be proved, this means that the {\color{black}supremum} of $I$ could always extend to a larger number. {\color{black}In other words}, 
$I$ is open. Since $0\in I$, $I$ is not an empty set. Therefore, we have $I$ is the whole $\mathbb{R}$. This gives that (\ref{bootstrap}) holds for all $t$.  
 \end{remark}
 
Now, under the assumption (\ref{bootstrap}), we start to derive (\ref{bootstrap2}). From (\ref{epsilon}) we have 
\begin{equation*}
\begin{split}
&\{1+\f{3\lambda^2\Gamma}{\Omega}\rb(0)^4 t\}\e(t)\\
=&\int_0^t\{\fc\} \rb(t')^4 \{\e(t')^2+\e(t')^3+\e(t')^4+\e(t')^5\} dt'\\
&+\int_0^t\{\fc\}\f{Re(e^{-i\theta(t')}e^{i\Omega t'}E(t'))}{\rb(t')} dt'\\
&+\int_0^t\{\fc\} \f{Re(M(e^{i\theta(t')}))}{\rb(t')}dt'\\
&+\int_0^t\{\fc\} \rb(t')^2(1+\e(t'))^3\sin2\theta(t') dt'\\
&+\int_0^t\{\fc\} \rb(t')^2(1+\e(t'))^3\sin4\theta(t') dt'\\
=&I+II+III+IV+V.
\end{split}
\end{equation*}\\

We now proceed to show that $$|I|+|II|+|III|+|IV|+|V|\leq 2\d_0^{\f12}\{1+\f{3\lambda^2\Gamma}{\Omega}\rb(0)^4 t\}.$$

\subsection{Estimates for $I$.}

For $I$, with the fact
$$\rb(t)^4=\f{\rho(0)^4}{1+\f{3\lambda^2\Gamma}{\Omega}\rho(0)^4 t},$$
we have
\begin{equation*}
\begin{split}
|I|=& |\int_0^t\{\fc\} \rb(t')^4 \{\e(t')^2+\e(t')^3+\e(t')^4+\e(t')^5\} dt'|\\
\leq& \d_0\rho(0)^4 t.
\end{split}
\end{equation*}

\subsection{Estimates for $II$}
If we choose \begin{equation*}
0< \rho(0)\leq\d_0  \quad \mbox{and}  \quad \|u_0\|_{(W^{2,2}\cap W^{2,1})\times (W^{1,2}\cap W^{1,1})}\leq \d_0, \footnote{The $(W^{2,2}\cap W^{2,1})\times (W^{1,2}\cap W^{1,1})$ norm is used as norm $X$ in \cite{SW} to obtain local existence results. See Theorem 3.1 and (3.7) in \cite{SW}.}
\end{equation*}
by (\ref{estimate E}) it follows
\begin{equation}\label{5.9}
|E(t)|\leq \delta^2_0 (1+t)^{-\f54-\delta}, \quad \mbox{where} \quad \d>0. 
\end{equation}
For $II$, with (\ref{5.9}) we then arrive at
\begin{equation*}
\begin{split}
|II|=&|\int_0^t\{\fc\}\f{Re(e^{-i\theta(t')}e^{i\Omega t'}E(t'))}{\rb(t')} dt'|\\
\leq& \int_0^t \{\fc\}\frac{\d_0}{(1+t')^{1+\d}}dt'\\
\leq& \{1+\f{3\lambda^2\Gamma}{4\Omega}\rb(0)^4 t\}\d_0.
\end{split}
\end{equation*}

\subsection{Estimates for $III$}

For $III$, we first recall the definition for $M(e^{i\theta(t)})$:

\begin{equation*}
\begin{split}
&M(e^{\it})\\
=&\r^5e^{4\it}+\r^5e^{2\it}+\r^5e^{-2\it}+\r^5e^{-4\it}+\r^5e^{-6\it}.
\end{split}
\end{equation*}
Hence, $III$ is a finite sum of terms
$$III_{1k}=\int_0^t\{\fc\} \rb(t')^4 (1+\epsilon(t'))^5  \sin k\t(t') dt'$$
or
$$III_{2k}=\int_0^t\{\fc\} \rb(t')^4 (1+\epsilon(t'))^5 \cos k\t(t') dt',$$
where $k \neq 0$.

We will employ integration by {\color{black}parts} to explore the oscillation nature of $III_{1k}$ and $III_{2k}$ and we will encounter $\theta'(t)$ and $\theta''(t)$ terms. Before moving into full details of integration by {\color{black}parts}, we first state two useful propositions. 

\begin{proposition}
For $\theta'(t)$ and $\theta''(t)$, we have
\begin{equation}\label{theta}
\begin{split}
&\theta'(t)-\Omega\\
=&-\frac{3\lambda}{2\Omega}\|\psi\|^4_4(1+\e(t))^2\rb(t)^2\\
&-\f{3\lambda^2}{4\Omega}[\Lambda-5\underline{\rho}(\Omega)+3\underline{\rho}(-\Omega)+\underline{\rho}(-3\Omega)](1+\e(t))^4\rb(t)^4\\
&-\frac{2\lambda}{\Omega}\|\psi\|^4_4 \cos{2\theta(t)}(1+\e(t))^2\rb(t)^2-\frac{\lambda}{2\Omega}\|\psi\|^4_4 \cos{4\theta(t)}(1+\e(t))^2\rb(t)^2\\
&+\f{1}{\rb(t)(1+\e(t))}Im(M(e^{\it}))+\f{1}{\rb(t)(1+\e(t))}Im(e^{-i\t(t)}e^{i\O t}E(t)),
\end{split}
\end{equation}
and

\begin{equation}\label{theta2}
\begin{split}
&\theta''(t)\\
=&(1+\e(t))\e'(t)\rb(t)^2+(1+\e(t))^2\rb(t)\rb'(t)\\
&+(1+\e(t))^3\e'(t)\rb(t)^4+(1+\e(t))^4\rb(t)^3\rb'(t)\\
&+\sin{2\theta(t)}\t'(t)(1+\e(t))^2\rb(t)^2+\cos{2\theta(t)}(1+\e(t))\e'(t)\rb(t)^2\\
&+\cos{2\theta(t)}(1+\e(t))^2\rb(t)\rb'(t)+\sin{4\theta(t)}\t'(t)(1+\e(t))^2\rb(t)^2\\
&+\cos{4\theta(t)}(1+\e(t))\e'(t)\rb(t)^2+\cos{4\theta(t)}(1+\e(t))^2\rb(t)\rb'(t)\\
&+\{\f{1}{\rb(t)(1+\e(t))}Im(M(e^{\it}))\}'\\
&+\{\f{1}{\rb(t)(1+\e(t))}Im(e^{-i\t(t)}e^{i\O t}E(t))\}'.
\end{split}
\end{equation}
\end{proposition} 
\begin{proof}
Recall (\ref{theta0}) 
\begin{equation*}
\begin{split}
&\theta'(t)\r-\Omega\r+\frac{3\lambda}{2\Omega}\|\psi\|^4_4\r^3\\
&+\f{3\lambda^2}{4\Omega}[\Lambda-5\underline{\rho}(\Omega)+3\underline{\rho}(-\Omega)+\underline{\rho}(-3\Omega)]\rho(t)^5\\
&+\frac{2\lambda}{\Omega}\|\psi\|^4_4 \cos{2\theta(t)}\r^3+\frac{\lambda}{2\Omega}\|\psi\|^4_4 \cos{4\theta(t)}\r^3\\
=&Im(M(e^{\it}))+Im(e^{-i\t(t)}e^{i\O t}E(t)).
\end{split}
\end{equation*}
Since $\rho(t)=\rb(t)(1+\e(t))$, we get
\begin{equation*}
\begin{split}
&\theta'(t)(1+\e(t))-\Omega(1+\e(t))+\frac{3\lambda}{2\Omega}\|\psi\|^4_4(1+\e(t))^3\rb(t)^2\\
&+\f{3\lambda^2}{4\Omega}[\Lambda-5\underline{\rho}(\Omega)+3\underline{\rho}(-\Omega)+\underline{\rho}(-3\Omega)](1+\e(t))^5\rb(t)^4\\
&+\frac{2\lambda}{\Omega}\|\psi\|^4_4 \cos{2\theta(t)}(1+\e(t))^3\rb(t)^2+\frac{\lambda}{2\Omega}\|\psi\|^4_4 \cos{4\theta(t)}(1+\e(t))^3\rb(t)^2\\
=&\f{1}{\rb(t)}Im(M(e^{\it}))+\f{1}{\rb(t)}Im(e^{-i\t(t)}e^{i\O t}E(t)).
\end{split}
\end{equation*}
Divide $1+\e(t)$ on both sides. We then obtain (\ref{theta}). Taking an additional derivative {\color{black}with} respect to $t$, (\ref{theta2}) follows. 
\end{proof}

\begin{proposition}
Under the bootstrap assumption (\ref{bootstrap}), and the assumptions
$$0< \rho(0)\leq\d_0,$$  
we have
\begin{equation}\label{t1b}
|\t'(t)-\O|\leq \d_0,
\end{equation}
\begin{equation}\label{t2b}
|\t''(t)|\leq {\color{black}\f{1}{(1+t)^{\f12}}}.
\end{equation}
\end{proposition}

\begin{proof}
Together with (\ref{rhob4}), (\ref{rhob}), (\ref{epsilon}), (\ref{5.9}), (\ref{theta2}) and bootstrap assumption (\ref{bootstrap}), for $III_{1k}$ term, it is straightforward to check 
\begin{equation*}
|\t'(t)-\O|\leq \d_0.
\end{equation*}

For $\theta''(t)$, by Proposition A.3 in appendix, we have
\begin{equation*}
|E'(t)|\leq \d^2_0 (1+t)^{-1}.
\end{equation*}
Recall

$$\rb(t)=\f{\rho(0)}{(1+\f{3\lambda^2\Gamma}{\Omega}\rho(0)^4 t)^{\f14}}\leq \f{1}{(1+t)^{\f14}}.$$
Together with (\ref{rhob4}), (\ref{rhob}), (\ref{epsilon}), (\ref{5.9}), (\ref{theta}) and bootstrap assumption (\ref{bootstrap}), this implies

\begin{equation*}
|\t''(t)|\leq {\color{black}\f{1}{(1+t)^{\f12}}}.
\end{equation*}
\end{proof}

We are now ready to analyze $III_{1k}$. With the help of integration by parts and the identity
$${(1+\f{3\lambda^2\Gamma}{\Omega}\rho(0)^4 t)}\cdot\rb(t)^4={\rho(0)}^{4},$$

we have

\begin{equation*}
\begin{split}
II&I_{1k}=\int_0^t\{\fc\} \rb(t')^4 (1+\epsilon(t'))^5\sin k\t(t') dt'\\
=&\int_0^t\rho(0)^4 \sin k\t(t') (1+\epsilon(t'))^5 dt'\\
=&\int_0^t\rho(0)^4 \f{-1}{k\t'(t')} (1+\epsilon(t'))^5 d \cos k\t(t')\\
=&\f{-\rho(t)^4}{k\t'(t)}\cos k\t(t)(1+\epsilon(t))^5+\f{\rho(0)^4}{k\theta'(0)}\cos k\theta(0)\\
&+\int_0^t\rho(0)^4\cos k\t(t') (1+\epsilon(t'))^5 d\f{1}{k\t'(t')}+\int_0^t \rho(0)^4 \f{\cos k\theta(t')}{k\theta'(t)}5(1+\epsilon(t'))^4 \epsilon'(t)dt' \\
=&\f{-\rho(t)^4}{k\t'(t)}\cos k\t(t)(1+\epsilon(t))^5+\f{\rho(0)^4}{k\theta'(0)}\cos k\theta(0)\\
&-\int_0^t\rho(0)^4\cos k\t(t') \f{1}{k(\t'(t'))^2}\t''(t')(1+\epsilon(t'))^5 dt'\\
&+5\int_0^t \rho(0)^4 \f{\cos k\theta(t')}{k\theta'(t)}(1+\epsilon(t'))^4 \epsilon'(t)dt'.\\
\end{split}
\end{equation*}
By (\ref{t1b}) and (\ref{t2b}), we conclude
$$|\f{-\rho(t)^4}{k\t'(t)}\cos k\t(t)|\leq \rho(0)^4,\quad |\f{\rho(0)^4}{k\theta'(0)}\cos k\theta(0)|\leq \rho(0)^4,$$
\begin{equation*}
\begin{split}
|\int_0^t\rho(0)^4\cos k\t(t') \f{1}{k(\t'(t'))^2}\t''(t') (1+\epsilon(t'))^5dt'|\leq &\rho(0)^4  (1+t)^{\f12}.
\end{split}
\end{equation*}
And by (\ref{epsilon}), we have $\epsilon'(t)\leq 1/(1+t)^{\f12}$. This implies
\begin{equation*}
\begin{split}
|\int_0^t \rho(0)^4 \f{\cos k\theta(t')}{k\theta'(t)}(1+\epsilon(t'))^4 \epsilon'(t)dt' |\leq &\rho(0)^4  (1+t)^{\f12}.
\end{split}
\end{equation*}
Now we verify that for $0\leq \d_0\ll \min\{\lambda, \Gamma, \f{1}{\O}\}$
\begin{equation}\label{1001}
\rho(0)^4 (1+t^{\f12})\ll \{1+\f{3\lambda^2\Gamma}{4\O}\rho(0)^4 t\}\d_0 \quad \mbox{for all} \quad t>0.
\end{equation}
We have three scenarios:

\begin{itemize}
\item When $t\leq 100$, since $\rho(0)\leq \d_0$, inequality (\ref{1001}) holds.
\item When $100\leq t\leq {1}/{\d^3_0}$, inequality (\ref{1001}) is also true.
\item When $t\geq {1}/{\d^3_0}$, we have
$\rho(0)^4 t^{\f12}\ll\f{3\lambda^2\Gamma}{4\O} \rho(0)^4 t\d_0.$ This implies (\ref{1001}).
\end{itemize}
Therefore, (\ref{1001}) holds for all cases and we arrive at
$$|III_{1k}|\ll \{1+\f{3\lambda^2\Gamma}{4\O}\rho(0)^4 t\}\d_0.$$
Similarly, $III_{2k}$ obeys the same bound
$$|III_{2k}|\ll \{1+\f{3\lambda^2\Gamma}{4\O}\rho(0)^4 t\}\d_0.$$
Since $III$ is a finite sum of $III_{1k}$ and $III_{2k}$, we hence deduce
$$|III|\ll \{1+\f{3\lambda^2\Gamma}{4\O}\rho(0)^4 t\}\d_0.$$

\subsection{Estimates for $IV$, $V$}

Now we move to term $IV$. For this term, we exploit integration by {\color{black}parts}:

\begin{equation*}
\begin{split}
IV=&\int_0^t\{\fc\} \rb(t')^2(1+\e(t'))^3\sin2\theta(t') dt'\\
=&\int_0^t\{\fc\} \rb(t')^2(1+\e(t'))^3\f{-1}{2\t'(t')} d\cos2\theta(t') dt'\\
=&\f{\rho(0)^2}{2\theta'(0)}\cos2\theta(0)+\{1+\f{3\lambda^2\Gamma}{\O}\rho(0)^4 t\}\rb(t)^2(1+\e(t))^3\f{-1}{2\t'(t)} \cos2\theta(t)\\
&+ \int_0^t \cos2\theta(t') d \Big\{ \{\fc\} \rb(t')^2(1+\e(t'))^3\f{-1}{2\t'(t')} \Big\}\\
=&\f{\rho(0)^2}{2\theta'(0)}\cos2\theta(0)+\{1+\f{3\lambda^2\Gamma}{\O}\rho(0)^4 t\}\rb(t)^2(1+\e(t))^3\f{-1}{2\t'(t)} \cos2\theta(t)\\
&+ \int_0^t \cos2\theta(t') \f{3\lambda^2\Gamma}{\O}\rho(0)^4 \rb(t')^2(1+\e(t'))^3\f{-1}{2\t'(t')} dt'\\
&+ \int_0^t \cos2\theta(t') \{\fc\} \rb(t')\rb'(t')(1+\e(t'))^3\f{-1}{2\t'(t')} dt'\\
&+ \int_0^t \cos2\theta(t') \{\fc\} \rb(t')^2(1+\e(t'))^2\e'(t')\f{-3}{2\t'(t')} dt'\\
&+ \int_0^t \cos2\theta(t') \{\fc\} \rb(t')^2(1+\e(t'))^3\f{2}{2(\t'(t'))^2} \t''(t') dt'\\
=&IV_0+IV_1+IV_2+IV_3+IV_4+IV_5.
\end{split}
\end{equation*}
With (\ref{rho}), (\ref{rhob4}), (\ref{t1b}) and {\color{black}the} bootstrap assumption (\ref{bootstrap}), we have

$$|IV_0|\leq \d^2_0,$$

$$|IV_1|\leq \{1+\f{3\lambda^2\Gamma}{4\O}\rho(0)^4 t\}\d^2_0,$$

\begin{equation*}
\begin{split}
|IV_2|\leq& \rho^4(0)\int_0^t \rb^2(t')dt'\\
=&\rho(0)^6\int_0^t\f{1}{\{1+\f{3\lambda^2\Gamma}{\Omega}\rho(0)^4 t'\}^{\f12}}dt'\\
\leq&\rho(0)^2\{1+\f{3\lambda^2\Gamma}{\Omega}\rho(0)^4 t\}^{\f12}\\
\leq&\{1+\f{3\lambda^2\Gamma}{\Omega}\rho(0)^4 t\}\d^2_0,\\
\end{split}
\end{equation*}

\begin{equation*}
\begin{split}
|IV_3|\leq & \int_0^t  \{\fc\} \rb(t')^6 dt'\\
\leq&\{1+\f{3\lambda^2\Gamma}{\Omega}\rho(0)^4 t\} \rho(0)^6\int_0^t\f{1}{\{1+\f{3\lambda^2\Gamma}{\Omega}\rho(0)^4 t'\}^{\f32}}dt'\\
\leq&\{1+\f{3\lambda^2\Gamma}{\Omega}\rho(0)^4 t\}\d^2_0.\\
\end{split}
\end{equation*}
We then move to next term

\begin{equation}\label{999}
IV_4=\int_0^t \cos2\theta(t') \{\fc\} \rb(t')^2(1+\e(t'))^2\e'(t')\f{-1}{2\t'(t')} dt'.
\end{equation}
Recall from (\ref{epsilon}), we have

\begin{equation}\label{1002}
\begin{split}
&\e'(t)\\
=&-\f{3\lambda^2\Gamma}{4\Omega}\rb(t)^4\e(t)+\rb(t)^4\e(t)^2+\rb(t)^4\e(t)^3+\rb(t)^4\e(t)^4+\rb(t)^4\e(t)^5\\
&+\rb(t)^2(1+\e(t))^3\sin2\theta(t)+\rb(t)^2(1+\e(t))^3\sin4\theta(t)\\
&+\f{1}{\rb(t)}Re(M(e^{i\theta(t)}))+\f{1}{\rb(t)}Re(e^{-i\theta(t)}e^{i\Omega t}E(t)).
\end{split}
\end{equation}
Plug (\ref{1002}) into (\ref{999}),  we have
$$IV_4=IV_{41}+IV_{42}+IV_{43}+IV_{44},$$
where $IV_{41}$ is sum of terms containing $\rb(t')^6$ in the integrand, and

\begin{equation*}
IV_{42}=\int_0^t \cos2\theta(t') \sin2\t(t') \{\fc\} \rb(t')^4(1+\e(t'))^5\f{-1}{2\t'(t')} dt',
\end{equation*}

\begin{equation*}
IV_{43}=\int_0^t \cos2\theta(t') \sin4\t(t') \{\fc\} \rb(t')^4(1+\e(t'))^5\f{-1}{2\t'(t')} dt',
\end{equation*}

\begin{equation*}
\begin{split}
IV_{44}=&\int_0^t \cos2\theta(t')\{\fc\} \rb(t')(1+\e(t'))^2\\
&\times Re(e^{-i\theta(t')}e^{i\Omega t'}E(t'))\f{-1}{2\t'(t')} dt'.
\end{split}
\end{equation*}
Employing angle difference identities, we derive

\begin{equation*}
IV_{42}=\int_0^t \sin4\t(t') \{\fc\} \rb(t')^4(1+\e(t'))^5\f{-1}{4\t'(t')} dt',
\end{equation*}
and
\begin{equation*}
\begin{split}
IV_{43}=&\int_0^t \sin6\t(t')\{\fc\} \rb(t')^4(1+\e(t'))^5\f{-1}{4\t'(t')} dt'\\
&+\int_0^t \sin2\t(t') \{\fc\} \rb(t')^4(1+\e(t'))^5\f{-1}{4\t'(t')} dt'.
\end{split}
\end{equation*}
For $IV_{41}$, using (\ref{rhob4}) we derive

\begin{equation*}
\begin{split}
|IV_{41}|\leq&\int_0^t \{\fc\} \rb(t')^6 dt'\\
\leq&\{1+\f{3\lambda^2\Gamma}{\Omega}\rho(0)^4 t\} \rho(0)^6\int_1^t \f{1}{\{1+\f{3\lambda^2\Gamma}{\Omega}\rho(0)^4 t'\}^{\f32}}dt'\\
\leq&\{1+\f{3\lambda^2\Gamma}{\Omega}\rho(0)^4 t\} \rho(0)^{2}\\
\leq&\{1+\f{3\lambda^2\Gamma}{\Omega}\rho(0)^4 t\} \d^2_0.\\
\end{split}
\end{equation*}
For $IV_{42}$, we employ integration by {\color{black}parts} once more:

\begin{equation*}
\begin{split}
 IV_{42}=&\int_0^t \sin4\t(t') \{\fc\} \rb(t')^4(1+\e(t'))^5\f{-1}{4\t'(t')} dt'\\
 =&\int_0^t \{\fc\} \rb(t')^4(1+\e(t'))^5\f{1}{16(\t'(t'))^2} d \cos4\t(t')\\
 =&\f{\rho(0)^4}{16(\theta'(0))^2}\cos4\theta(0)\\
 &+\{1+\f{3\lambda^2\Gamma}{\Omega}\rho(0)^4 t\}\rb(t)^4(1+\e(t))^5\f{1}{16(\t'(t))^2}\cos 4\t(t)\\
 &+\int_0^t \f{3\lambda^2\Gamma}{\Omega}\rho(0)^4 \rb(t')^4(1+\e(t'))^5\f{1}{16(\t'(t'))^2} \cos4\t(t') dt'\\
 &+\int_0^t \{\fc\} \rb(t')^3\rb'(t')(1+\e(t'))^5\f{1}{16(\t'(t'))^2} \cos4\t(t') dt'\\
 &+\int_0^t \{\fc\} \rb(t')^4(1+\e(t'))^4\e'(t')\f{1}{16(\t'(t'))^2} \cos4\t(t') dt'\\
 &-\int_0^t \{\fc\} \rb(t')^4(1+\e(t'))^5\f{\t''(t')}{8(\t'(t'))^3} \cos4\t(t')dt'\\
 =&IV_{420}+IV_{421}+IV_{422}+IV_{423}+IV_{424}+ IV_{425}.
\end{split}
\end{equation*}
With  (\ref{rhob4}), (\ref{t1b}), bootstrap assumption (\ref{bootstrap}), we derive

$$|IV_{420}|\leq \d^4_0,$$

$$|IV_{421}|\leq \{1+\f{3\lambda^2\Gamma}{\Omega}\rho(0)^4 t\}\d^4_0,$$

\begin{equation*}
\begin{split}
|IV_{422}|\leq& \rho(0)^8\int_0^t \f{1}{\{1+\f{3\lambda^2\Gamma}{\Omega}\rho(0)^4 t'\}}dt'\\
\leq& \rho(0)^8\int_0^t \f{1}{\{1+\f{3\lambda^2\Gamma}{\Omega}\rho(0)^4 t'\}^{\f12}}dt'\\
\leq& \rho(0)^4\{1+\f{3\lambda^2\Gamma}{\Omega}\rho(0)^4 t\}^{\f12}\\
\leq&\{1+\f{3\lambda^2\Gamma}{\Omega}\rho(0)^4 t\}\d^2_0.
\end{split}
\end{equation*}
Together with (\ref{rhob}), (\ref{epsilon}) and (\ref{t2b}), we derive

\begin{equation*}
\begin{split}
|IV_{423}|=&|\int_0^t \{\fc\} \rb(t')^3\rb'(t')(1+\e(t'))^5\f{1}{16(\t'(t'))^2} \cos4\t(t') dt'|\\
\leq&\{1+\f{3\lambda^2\Gamma}{\Omega}\rho(0)^4 t\}\int_0^t \rb(t')^8 dt'\\
\leq&\{1+\f{3\lambda^2\Gamma}{\Omega}\rho(0)^4 t\} \rho(0)^8\int_0^t \f{1}{\{1+\f{3\lambda^2\Gamma}{\Omega}\rho(0)^4 t'\}^2}dt'\\
\leq&\{1+\f{3\lambda^2\Gamma}{\Omega}\rho(0)^4 t\} \rho(0)^4\\
\leq&\{1+\f{3\lambda^2\Gamma}{\Omega}\rho(0)^4 t\} \d^4_0,\\
\end{split}
\end{equation*}

\begin{equation*}
\begin{split}
|IV_{424}|=&|\int_0^t \{\fc\} \rb(t')^4(1+\e(t'))^4\e'(t')\f{\cos4\t(t')}{16(\t'(t'))^2}  dt'|\\
\leq&\{1+\f{3\lambda^2\Gamma}{\Omega}\rho(0)^4 t\} \rho(0)^6\int_0^t \f{1}{\{1+\f{3\lambda^2\Gamma}{\Omega}\rho(0)^4 t'\}^{\f32}}dt'\\
&+\{1+\f{3\lambda^2\Gamma}{\Omega}\rho(0)^4 t\} \rho(0)^{3}\int_0^t \d^2_0 (1+t')^{-\f54-\delta}dt'\\
\leq&\{1+\f{3\lambda^2\Gamma}{\Omega}\rho(0)^4 t\} \rho(0)^{2}\\
\leq&\{1+\f{3\lambda^2\Gamma}{\Omega}\rho(0)^4 t\} \d^2_0,\\
\end{split}
\end{equation*}

\begin{equation*}
\begin{split}
|IV_{425}|=&|\int_0^t \{\fc\} \rb(t')^4(1+\e(t'))^5\f{\t''(t')}{8(\t'(t'))^3} \cos4\t(t')dt'|\\
\leq&\int_0^t \{\fc\} \rb(t')^4 \f{{\color{black}1}}{(1+t')^{\f12}} dt'\\
\leq&{\color{black}\rho(0)^4\int_0^t\f{1}{(1+t')^{\f12}}dt'\leq\rho(0)^4(1+t^{\f12})}\\
\leq&\{1+\f{3\lambda^2\Gamma}{\Omega}\rho(0)^4 t\} \d_0 \color{black}{\mbox{ by} \,\, (\ref{1001})}.\\
\end{split}
\end{equation*}
Combine all the bounds for $IV_{421}-IV_{425}$ together, we prove
$$|IV_{42}|\leq \{1+\f{3\lambda^2\Gamma}{\Omega}\rho(0)^4 t\} \d_0.$$
For $IV_{43}$, using the same method as for $IV_{42}$, we arrive at
$$|IV_{43}|\leq \{1+\f{3\lambda^2\Gamma}{\Omega}\rho(0)^4 t\} \d_0.$$
For $IV_{44}$, it works as $IV_{42}$ with an extra estimate $|E(t)|\leq \d^2_0 (1+t)^{-\f54-\delta}$ ($\delta>0$) used. We thus get
$$|IV_{44}|\leq \{1+\f{3\lambda^2\Gamma}{\Omega}\rho(0)^4 t\} \d^2_0.$$
Therefore, we deduce
$$|IV_{4}|\leq \{1+\f{3\lambda^2\Gamma}{\Omega}\rho(0)^4 t\} \d_0.$$
{\color{black}For 
$$IV_5:=\int_0^t \cos2\theta(t') \{\fc\} \rb(t')^2(1+\e(t'))^3\f{2}{2(\t'(t'))^2} \t''(t') dt',$$
using (\ref{theta2})-the expression of  $\theta''(t)$, we notice that the borderline terms in $IV_5$    
$$\int_0^t \cos2\theta(t') \{\fc\} \rb(t')^2(1+\e(t'))^3\f{2}{2(\t'(t'))^2}\cdot \rb(t')^2(1+\e(t'))^2\t'(t')\sin2\t(t')dt',$$
$$\int_0^t \cos2\theta(t') \{\fc\} \rb(t')^2(1+\e(t'))^3\f{2}{2(\t'(t'))^2}\cdot \rb(t')^2(1+\e(t'))^2\t'(t')\sin4\t(t')dt',$$
are exactly the same terms (up to a constant) we encountered in $IV_4$:
\begin{equation*}
IV_{42}:=\int_0^t \cos2\theta(t') \sin2\t(t') \{\fc\} \rb(t')^4(1+\e(t'))^5\f{-1}{2\t'(t')} dt',
\end{equation*}

\begin{equation*}
IV_{43}:=\int_0^t \cos2\theta(t') \sin4\t(t') \{\fc\} \rb(t')^4(1+\e(t'))^5\f{-1}{2\t'(t')} dt'.
\end{equation*}
In the same fashion as to bound $IV_{4}$, we have}
$$|IV_{5}|\leq \{1+\f{3\lambda^2\Gamma}{\Omega}\rho(0)^4 t\} \d_0.$$
Gather all the bounds for $IV_{1}-IV_{5}$, we have proved
$$|IV|\leq \{1+\f{3\lambda^2\Gamma}{\Omega}\rho(0)^4 t\} \d_0.$$
In the same manner, we bound $V$ term
$$|V|\leq \{1+\f{3\lambda^2\Gamma}{\Omega}\rho(0)^4 t\} \d_0.$$\\

\subsection{Conclusion of bootstrap argument}

Combining the bounds for $I-IV$, we have proved
$$\{1+\f{3\lambda^2\Gamma}{\Omega}\rho(0)^4 t\}|\e(t)|\leq 2\d^{\f12}_0 \{1+\f{3\lambda^2\Gamma}{\Omega}\rho(0)^4 t\}.$$
This implies
$$|\e(t)|\leq 2\d^{\f12}_0,$$
which is an improvement of bootstrap assumption \ref{bootstrap}
$$|\e(t)|\leq 3\d^{\f12}_0.$$
Hence, we have showed for all time $t\geq 0$
$$|\e(t)|\leq 3\d^{\f12}_0.$$
Together with the ansatz

$$\rho(t)=\rb(t)(1+\e(t)),$$
we conclude that

$$\f12 \rb(t)\leq \rho(t) \leq \f32 \rb(t).$$
Recall (\ref{rhob4})
\begin{equation*}
\rb(t)^4=\f{\rho(0)^4}{1+\f{3\lambda^2\Gamma}{\Omega}\rho(0)^4 t},
\end{equation*}
therefore we derived both upper and lower bounds for $\rho(t)$. And we have proved Theorem \ref{thm1.2}:
\begin{equation*}
\f{\f12\rho(0)}{\bigg( 1+\f{3\lambda^2\Gamma}{\O}\rho(0)^4 t \bigg)^{\f14} } \leq \rho(t) \leq \f{\f32\rho(0)}{\bigg( 1+\f{3\lambda^2\Gamma}{\O}\rho(0)^4 t \bigg)^{\f14} }.
\end{equation*}
\begin{remark}
The constants $1/2$ and $3/2$ could be improved to $1-2\d_0^{\f12}$ and $1+2\d_0^{\f12}$, respectively. 
\end{remark}

\section{Forward $H^1$ Scattering}
For $\eta$, from (\ref{eta}) to (\ref{eta3}) we have

\begin{equation*}
\eta(t,x)=\eta_1(t,x)+\eta_2(t,x)+\eta_3(t,x),  \quad \mbox{where}
\end{equation*}

\begin{equation*}
(\partial^2_t+B^2)\eta_1=0, \quad \eta_1(0,x)=P_c u_0, \quad \partial_t \eta_1(0,x)=P_c u_1,
\end{equation*}

\begin{equation*}
(\partial^2_t+B^2)\eta_2=\lambda a^3 P_c \psi^3, \quad \eta_2(0,x)=0, \quad \partial_t \eta_2(0,x)=0,
\end{equation*}

\begin{equation*}
(\partial^2_t+B^2)\eta_3=\lambda P_c(3a^2\psi^2\eta+3a\psi\eta^2+\eta^3), \quad \eta_3(0,x)=0, \quad \partial_t \eta_3(0,x)=0.
\end{equation*}
These equations imply

\begin{equation*}
\eta_1(t,x)=\cos Bt P_c u_0+\f{\sin Bt}{B} P_c u_1,
\end{equation*}

\begin{equation}\label{eta26}
\eta_2(t,x)=\lambda\int_0^t \f{\sin B(t-s)}{B} a^3P_c \psi^3 ds,
\end{equation}

\begin{equation}\label{eta36}
\eta_3(t,x)=\lambda\int_0^t \f{\sin B(t-s)}{B} P_c (3a^2\psi^2\eta+3a\psi\eta^2+\eta^3) ds.
\end{equation}
For the rest of this paper, we will show that
$$\eta_2(t,x)=\mbox{free wave}+R_2(t,x), \quad \quad \eta_3(t,x)=\mbox{free wave}+R_3(t,x),$$
where

$$\lim_{t\rightarrow +\infty} \|R_2(t,x)\|_{H^1_x}=0, \quad \mbox{and} \quad \lim_{t\rightarrow +\infty} \|R_3(t,x)\|_{H^1_x}=0.$$
We deal with $\eta_2$ and $\eta_3$ separately.

\subsection{The term $\eta_2(t,x)$}
For $\eta_2(t,x)$, we have

\begin{equation*}
\begin{split}
&\eta_2(t,x)\\
=&\lambda\int_0^t \f{\sin B(t-s)}{B} a^3(s)P_c \psi^3 ds\\
=&\lambda\int_0^t (\f{\sin Bt \cos Bs}{B}-\f{\cos Bt \sin Bs}{B}) a^3(s)P_c \psi^3 ds\\
=&\lambda \f{\sin Bt}{B} \int_0^t \cos Bs \,a^3(s)P_c \psi^3 ds-\lambda \cos Bt \int_0^t \f{\sin Bs}{B} a^3(s)P_c \psi^3 ds\\
=&\lambda \f{\sin Bt}{B} \int_0^{+\infty} \cos Bs \,a^3(s)P_c \psi^3 ds-\lambda \cos Bt \int_0^{+\infty} \f{\sin Bs}{B} a^3(s)P_c \psi^3 ds\\
&-\lambda \f{\sin Bt}{B} \int_t^{+\infty} \cos Bs \,a^3(s)P_c \psi^3 ds+\lambda \cos Bt \int_t^{+\infty} \f{\sin Bs}{B} a^3(s)P_c \psi^3 ds.\\
\end{split}
\end{equation*}
In the following, we will first prove
\begin{proposition}\label{6.1}
For $a(t)$ and $\psi(t,x)$ as in Sections 1-5, we have
$$\int_0^{+\infty} \cos Bs \,a^3(s)P_c \psi^3 ds \in L^2_x, \quad \int_0^{+\infty} \sin Bs \,a^3(s)P_c \psi^3 ds \in L^2_x,$$  
$$\int_0^{+\infty} \f{\sin Bs}{B} a^3(s)P_c \psi^3 ds \in L^2_x, \quad \mbox{and}$$

$$\|\int_t^{+\infty} \cos Bs \,a^3(s)P_c \psi^3 ds\|_{L^2_x}\rightarrow 0, \quad \|\int_t^{+\infty} \sin Bs \,a^3(s)P_c \psi^3 ds\|_{L^2_x}\rightarrow 0, $$
$$\|\int_t^{+\infty} \f{\sin Bs}{B} a^3(s)P_c \psi^3 ds\|_{L^2_x}\rightarrow 0, \quad \mbox{as} \quad t\rightarrow +\infty.$$

\end{proposition} 

\begin{remark} These conclusions {\color{black}cannot be deduced} directly from dispersive estimates for $\cos Bs$ and $\sin Bs/ B$. This is because for $a(t)$ $(t>0)$, we only have $|a(t)|\approx 1/(1+t)^{\f14}$. 
And this slow decay rate is not enough to prove Proposition \ref{6.1}. In the below, we will introduce a new approach by constructing auxiliary functions and by exploring structures of the nonlinear Klein-Gordon equation. 
\end{remark}

Proposition \ref{6.1} is a crucial middle step. Let's first show that conclusions in Proposition \ref{6.1} imply the main scattering results.  Since $\sin Bt / B$ and $\cos Bt$ are bounded operators for $L^2_x$, the conclusions in Proposition \ref{6.1} yield
\begin{proposition}\label{new 6.2}
For $a(t)$ and $\psi(t,x)$ as in Sections 1-5, we have
$$\|-\lambda \f{\sin Bt}{B}\int_t^{+\infty} \cos Bs \,a^3(s)P_c \psi^3 ds\|_{L^2_x}\rightarrow 0,$$
$$\|\lambda \cos Bt \int_t^{+\infty} \f{\sin Bs}{B} a^3(s)P_c \psi^3 ds\|_{L^2_x}\rightarrow 0,$$
as $t\rightarrow +\infty$.
\end{proposition}
We can further improve this proposition to 
\begin{proposition}
For $a(t)$ and $\psi(t,x)$ as in Sections 1-5, we have
$$\|-\lambda \f{\sin Bt}{B}\int_t^{+\infty} \cos Bs \,a^3(s)P_c \psi^3 ds\|_{H^1_x}\rightarrow 0,$$
$$\|\lambda \cos Bt \int_t^{+\infty} \f{\sin Bs}{B} a^3(s)P_c \psi^3 ds\|_{H^1_x}\rightarrow 0,$$
as $t\rightarrow +\infty$.
\end{proposition}
\begin{proof}
Using Proposition \ref{6.1}, we also have as $t\rightarrow +\infty$
\begin{equation}\label{new eqn 6.5}
\begin{split}
&\|B\cdot \l -\lambda \f{\sin Bt}{B}\int_t^{+\infty} \cos Bs \,a^3(s)P_c \psi^3 ds \rr \|_{L^2_x}\\
=&\|-\lambda \sin Bt \int_t^{+\infty} \cos Bs \,a^3(s)P_c \psi^3 ds\|_{L^2_x}\rightarrow 0,\\
&\| B\cdot \l  \lambda \cos Bt \int_t^{+\infty} \f{\sin Bs}{B} a^3(s)P_c \psi^3 ds   \rr \|_{L^2_x}\\
=&\|\lambda \cos Bt \int_t^{+\infty} \sin Bs a^3(s)P_c \psi^3 ds\|_{L^2_x}\rightarrow 0.
\end{split}
\end{equation}
For any $f\in H^1_x$, since $B^2=-\Delta+V(x)+m^2$, we have
$$\int_{\mathbb{R}^3} \l |\nab f|^2+V(x)f^2+m^2 f^2 \rr dx=\|Bf\|^2_{L^2_x}. $$
Using the assumption that $\|V(x)\|_{L^{\infty}_x}$ is bounded by a constant $C$, we arrive at 
\begin{equation}\label{B lemma}
\begin{split}
\|f\|^2_{\dot{H}^1_x}\leq \|Bf\|^2_{L^2_x}+C\|f\|^2_{L^2_x}.
\end{split}
\end{equation}
Together with Proposition \ref{new 6.2} and (\ref{new eqn 6.5}), we deduce the conclusions in this proposition.\\
\end{proof}

We notice that
$$\tilde{u}(t,x):=\lambda \f{\sin Bt}{B} \int_0^{+\infty} \cos Bs \,a^3(s)P_c \psi^3(x) ds-\lambda \cos Bt \int_0^{+\infty} \f{\sin Bs}{B} a^3(s)P_c \psi^3(x) ds$$
is a free wave. (It satisfies $\partial^2_{t} \tilde{u}(t,x)+B^2 \tilde{u}(t,x)=0$.{\color{black})} Therefore, once Proposition \ref{6.1} is proved, we will arrive at
\begin{proposition}\label{6.2}
For $\eta_2$ is as in (\ref{eta26}), we have
$$\eta_2(t,x)=\mbox{free wave}+R_2(t,x), \,\mbox{where}\, \lim_{t\rightarrow +\infty} \|R_2(t,x)\|_{H^1_x}=0,$$
\end{proposition}

We now move back to prove Proposition \ref{6.1}. We will verify it by establishing two lemmas. Let's move to the first one:

\begin{lemma}\label{lemma6.3}
For $a(t)$ and $\psi(t,x)$ as in Sections 1-5, we have
\begin{equation}\label{eq613}
\int_0^t \sin B(t-s)\,\lambda a^3(s)P_c \psi^3(x) ds \in L^2_x,
\end{equation}

\begin{equation}\label{eq614}
\int_0^t \cos B(t-s)\,\lambda a^3(s)P_c \psi^3(x) ds \in L^2_x,
\end{equation}
for any $t>0$. 

\end{lemma}
\begin{proof}
We construct an auxiliary complex-valued function $w(t,x)$:
$$w(t,x)=\int_0^t \cos B(t-s)\,\lambda a^3(s)P_c \psi^3(x) ds-i\cdot \int_0^t \sin B(t-s)\,\lambda a^3(s)P_c \psi^3(x) ds.$$
Thus,
\begin{equation}\label{Im}
\mbox{Im}\,w(t,x)=-\int_0^t \sin B(t-s)\,\lambda a^3(s)P_c \psi^3(x) ds,
\end{equation}
and
\begin{equation}\label{Re}
\mbox{Re}\,w(t,x)=\int_0^t \cos B(t-s)\,\lambda a^3(s)P_c \psi^3(x) ds.
\end{equation}
From standard dispersive estimates (see Theorem 2.1 in \cite{SW}), we have
$$\mbox{Im}\,w(t,x), \, \mbox{Re}\,w(t,x) \in L^8_x \, \mbox{for all}\, t,$$
For $t$ close to 0, it is straightforward that
$$\mbox{Im}\,w(t,x), \, \mbox{Re}\,w(t,x) \in L^2_x.$$
In the rest part of this lemma, we want to show that
$$\mbox{Im}\,w(t,x), \, \mbox{Re}\,w(t,x) \in L^2_x \,\, \mbox{for all}\, t.$$

We first claim that $w(t,x)$ actually solves

\begin{equation}\label{eq65}
\begin{split}
(i\partial_t +B)w=&\lambda a^3 P_c \psi^3,\\
w(0,x)&=0.
\end{split}
\end{equation}
That is
\begin{equation*}
w(t,x)=-i\int_0^t e^{iB(t-s)}\lambda a^3(s)P_c \psi^3(x) ds. \quad 
\end{equation*}
We could verify this claim by rewriting the above expression and checking its real and imaginary parts:

\begin{equation*}
\begin{split}
w(t,x)=&-i\int_0^t e^{iB(t-s)}\lambda a^3(s)P_c \psi^3(x) ds\\
=&-i\int_0^t \sin B(t-s)\,\lambda a^3(s)P_c \psi^3(x) ds\\
&+\int_0^t \cos B(t-s)\, \lambda a^3(s)P_c\psi^3(x)ds.
\end{split}
\end{equation*}
Here,
\begin{equation}\label{Im}
\mbox{Im}\,w(t,x)=-\int_0^t \sin B(t-s)\,\lambda a^3(s)P_c \psi^3(x) ds,
\end{equation}
and
\begin{equation}\label{Re}
\mbox{Re}\,w(t,x)=\int_0^t \cos B(t-s)\,\lambda a^3(s)P_c \psi^3(x) ds.
\end{equation}
Therefore, we prove our claim: $\omega(t,x)$ is a solution to (\ref{eq65}).

Next we define
\begin{equation*}
l(t):=\|w(t,x)\|^2_{L^2_x}=\|\mbox{Re}\,w(t,x)\|^2_{L^2_x}+\|\mbox{Im}\,w(t,x)\|^2_{L^2_x}.
\end{equation*}
We then check $\f{d}{dt}l(t)$. Since $B$ is a self-adjoint operator, we have

\begin{equation}\label{dl}
\begin{split}
\f{d}{dt}l(t)=&\f{d}{dt}\int_{\mathbb R^3} w(t,x)\bar{w}(t,x) dx\\
=&\int_{\mathbb R^3} \f{d}{dt}w(t,x)\bar{w}(t,x) dx+\int_{\mathbb R^3} w(t,x)\f{d}{dt}\bar{w}(t,x) dx\\
=&\int_{\mathbb R^3} iB w(t,x)\bar{w}(t,x)-i\lambda a^3(t)P_c\psi^3(x)\bar{w}(t,x) dx\\
&+\int_{\mathbb R^3} w(t,x)\cdot -iB\bar{w}(t,x)+w(t,x)i\lambda a^3(t)P_c\psi^3(x) dx\\
=&i\int_{\mathbb R^3} \lambda a^3(t)P_c\psi^3(x) (w(t,x)-\bar{w}(t,x)) dx\\
=&-2\int_{\mathbb R^3} \mbox{Im}\,w(t,x)\lambda a^3(t)P_c\psi^3(x)dx.
\end{split}
\end{equation}
Using the definitions of $\mbox{Re}\, w(t,x), \, \mbox{Im}\, w(t,x)$ in (\ref{Im}) and (\ref{Re}), together with the dispersive estimates for $B$ (see Theorem 2.1 in 
\cite{SW}) and the fact $|a(t)|\leq 1/(1+t)^{\f14}$, we derive
\begin{equation*}
\begin{split}
&\|\mbox{Re}\, w(t,x), \, \mbox{Im}\, w(t,x)\|_{L^8_x}\\
\leq&  \int_0^t (1+t-s)^{-\f98}(1+s)^{-\f34}ds\\
\leq& \int_0^{t/2}(1+t-s)^{-\f98}(1+s)^{-\f34}ds+ \int_{t/2}^t (1+t-s)^{-\f98}(1+s)^{-\f34}ds\\
\leq& \f{1}{(1+t)^{\f34}}.
\end{split}
\end{equation*}
With the estimate above, we arrive at
\begin{equation*}
\begin{split} 
l(t)\leq& l(0)+\int_0^t |\f{d}{dt}l(t)|dt\\
\leq&\int_0^t \|\mbox{Im}\, w(t,x)\|_{L^8_x} |a^3(t)| \|P_c\psi^3(x)\|_{L^{\f87}_x} dt\\
\leq& \int_0^t \f{1}{(1+t)^{\f34}}\cdot\f{1}{(1+t)^{\f34}}<+\infty.
\end{split}
\end{equation*}
Therefore, we have proved 

\begin{equation}\label{eq613}
\mbox{Im}\,w(t,x)=\int_0^t \sin B(t-s)\,\lambda a^3(s)P_c \psi^3(x) ds \in L^2_x,
\end{equation}

\begin{equation}\label{eq614}
\mbox{Re}\,w(t,x)=\int_0^t \cos B(t-s)\,\lambda a^3(s)P_c \psi^3(x) ds \in L^2_x,
\end{equation}
for any $t>0$. \\

\end{proof}
We then proceed to prove the second lemma.
\begin{lemma}\label{lemma6.4}
For $a(t)$ and $\psi(t,x)$ as in Sections 1-5, we have

\begin{equation}\label{eq615}
\|\int_t^{+\infty} \sin B(t-s)\,\lambda a^3(s)P_c \psi^3(x) ds\|^2_{L^2_x}\rightarrow 0,
\end{equation}

\begin{equation}\label{eq616}
\|\int_t^{+\infty} \cos B(t-s)\,\lambda a^3(s)P_c \psi^3(x) ds\|^2_{L^2_x}\rightarrow 0, \quad \mbox{as} \quad t\rightarrow +\infty.
\end{equation}

\end{lemma}

\begin{proof}
Construct an auxiliary function $\underline{w}$ through solving the following equation backward

\begin{equation}\label{eq65}
\begin{split}
(i\partial_t +B)\wb=&\lambda a^3 P_c \psi^3,\\
\wb(+\infty,x)&=0.
\end{split}
\end{equation}
Hence
\begin{equation*}
\underline{w}(t,x)=-i\int_{+\infty}^t e^{iB(t-s)}\lambda a^3(s)P_c \psi^3(x) ds. 
\end{equation*}
We further rewrite (\ref{eq65}) into

\begin{equation*}
\begin{split}
\wb(t,x)=&-i\int_{+\infty}^t e^{iB(t-s)}\lambda a^3(s)P_c \psi^3(x) ds\\
=&-i\int_{+\infty}^t \sin B(t-s)\,\lambda a^3(s)P_c \psi^3(x) ds\\
&+\int_{+\infty}^t \cos B(t-s)\, \lambda a^3(s)P_c\psi^3(x)ds.
\end{split}
\end{equation*}
Thus,
\begin{equation}\label{Imwb}
\mbox{Im}\,\wb(t,x)=\int^{+\infty}_t \sin B(t-s)\,\lambda a^3(s)P_c \psi^3(x) ds,
\end{equation}
and
\begin{equation}\label{Rewb}
\mbox{Re}\,\wb(t,x)=-\int^{+\infty}_t \cos B(t-s)\,\lambda a^3(s)P_c \psi^3(x) ds.
\end{equation}
Define

\begin{equation*}
{\underline{l}}(t):=\|\wb(t,x)\|^2_{L^2_x}=\|\mbox{Re}\,\wb(t,x)\|^2_{L^2_x}+\|\mbox{Im}\,\wb(t,x)\|^2_{L^2_x}.
\end{equation*}
We then check $\f{d}{dt}\underline{l}(t)$. Since $B$ is a self-adjoint operator, in the same manner as for (\ref{dl}), we derive

\begin{equation*}
\begin{split}
\f{d}{dt}\underline{l}(t)=-2\int_{\mathbb R^3} \mbox{Im}\,\wb(t,x)\lambda a^3(t)P_c\psi^3(x)dx.
\end{split}
\end{equation*}
Using the definitions of $\mbox{Re}\, \wb(t,x), \, \mbox{Im}\, \wb(t,x)$ in (\ref{Imwb}) and (\ref{Rewb}), together with the dispersive estimates for $B$ (see Theorem 2.1 in 
\cite{SW}) and the fact $|a(t)|\leq 1/(1+t)^{\f14}$, we derive
\begin{equation*}
\begin{split}
&\|\mbox{Re}\, \wb(t,x), \, \mbox{Im}\, \wb(t,x)\|_{L^8_x}\\
\leq&  \int_t^{+\infty} (1+s-t)^{-\f98}(1+s)^{-\f34}ds\\
\leq& \f{1}{(1+t)^{\f34}}.
\end{split}
\end{equation*}
Therefore, for any $t \rightarrow +\infty$, we have 
\begin{equation*}
\begin{split} 
|\underline{l}(t)|\leq&\int_{t}^{+\infty} |\f{d}{ds}\underline{l}(s)|ds\\
\leq&\int_{t}^{+\infty} \|\mbox{Im}\, \wb(s,x)\|_{L^8_x} |a^3(s)| \|P_c\psi^3(x)\|_{L^{\f87}_x} ds\\
\leq& \int_{t}^{+\infty} \f{1}{(1+s)^{\f34}}\cdot\f{1}{(1+s)^{\f34}}\\
\rightarrow& 0. 
\end{split}
\end{equation*}
Thus, we conclude that 

\begin{equation}\label{eq615}
\|\mbox{Im}\, \wb(t,x)\|_{L^2_x}=\|\int_t^{+\infty} \sin B(t-s)\,\lambda a^3(s)P_c \psi^3(x) ds\|^2_{L^2_x}\rightarrow 0,
\end{equation}

\begin{equation}\label{eq616}
\|\mbox{Re}\, \wb(t,x)\|_{L^2_x}=\|\int_t^{+\infty} \cos B(t-s)\,\lambda a^3(s)P_c \psi^3(x) ds\|^2_{L^2_x}\rightarrow 0, \quad \mbox{as} \quad t\rightarrow +\infty.
\end{equation}
We have hence finished the proof Lemma \ref{lemma6.4}. \\

\end{proof} 

We then continue to prove Proposition \ref{6.1}. Recall that $B^2=-\Delta+V(x)+m^2$ has continuous spectrum $\sigma_{cont}(B^2)=[m^2, +\infty)$ and a unique strictly positive single eigenvalue $\Omega^2<m^2$. 
This implies $1/B$ is a bounded operator for $L^2_x$. Together with the conclusions in Lemma \ref{lemma6.3} and Lemma \ref{lemma6.4}, it follows

\begin{equation}\label{eq617}
\int_0^t \f{\sin B(t-s)}{B}\,\lambda a^3(s)P_c \psi^3(x) ds \in L^2_x,
\end{equation}

\begin{equation}\label{eq618}
\int_0^t \f{\cos B(t-s)}{B}\,\lambda a^3(s)P_c \psi^3(x) ds \in L^2_x,
\end{equation}
for any $t>0$. And

\begin{equation}\label{619}
\|\int_t^{+\infty} \f{\sin B(t-s)}{B}\,\lambda a^3(s)P_c \psi^3(x) ds\|^2_{L^2_x}\rightarrow 0,
\end{equation}

\begin{equation}\label{620}
\|\int_t^{+\infty} \f{\cos B(t-s)}{B}\,\lambda a^3(s)P_c \psi^3(x) ds\|^2_{L^2_x}\rightarrow 0, \quad \mbox{as} \quad t\rightarrow +\infty.
\end{equation}
From angle-difference identities, we have

\begin{equation}\label{621}
\begin{split}
&\int_0^t \sin B(t-s)\,\lambda a^3(s)P_c \psi^3(x) ds\\
=&\sin Bt\int_0^t \cos Bs\, \lambda a^3 P_c\psi^3 ds- \cos Bt\int_0^t \sin Bs\, \lambda a^3 P_c\psi^3 ds \\
\end{split}
\end{equation}

\begin{equation}\label{622}
\begin{split}
&\int_0^t \cos B(t-s)\,\lambda a^3(s)P_c \psi^3(x) ds\\
=&\sin Bt\int_0^t \sin Bs\, \lambda a^3 P_c\psi^3 ds+ \cos Bt\int_0^t \cos Bs\, \lambda a^3 P_c\psi^3 ds \\
\end{split}
\end{equation}
Together with (\ref{eq613}) and (\ref{eq614}), we obtain

\begin{equation}\label{eq623}
\begin{split}
\sin Bt\int_0^t \cos Bs\, \lambda a^3 P_c\psi^3 ds- \cos Bt\int_0^t \sin Bs\, \lambda a^3 P_c\psi^3 ds \in L^2_x,
\end{split}
\end{equation}
and
\begin{equation}\label{eq624}
\begin{split}
\sin Bt\int_0^t \sin Bs\, \lambda a^3 P_c\psi^3 ds+ \cos Bt\int_0^t \cos Bs\, \lambda a^3 P_c\psi^3 ds \in L^2_x.
\end{split}
\end{equation}
Since $\sin Bt, \cos Bt$ are bounded operators for $L^2_x$,
$$\sin Bt\times(\ref{eq623})+\cos Bt\times (\ref{eq624})\in L^2_x$$
and it gives
\begin{equation}\label{eq625}
\begin{split}
&\sin Bt\,\sin Bt\int_0^t \cos Bs\, \lambda a^3 P_c\psi^3 ds- \sin Bt\,\cos Bt\int_0^t \sin Bs\, \lambda a^3 P_c\psi^3 ds\\
&+\cos Bt\, \sin Bt\int_0^t \sin Bs\, \lambda a^3 P_c\psi^3 ds+ \cos Bt\,\cos Bt\int_0^t \cos Bs\, \lambda a^3 P_c\psi^3 ds\\
=&\int_0^t \cos Bs\, \lambda a^3(s)P_c \psi^3(x)ds \in L^2_x.
\end{split}
\end{equation}

In the same fashion, we have
$$-\cos Bt\times(\ref{eq623})+\sin Bt\times (\ref{eq624})\in L^2_x,$$
and it gives
$$\int_0^t \sin Bs\, \lambda a^3(s)P_c \psi^3(x)ds \in L^2_x.$$
Similarly, angle-difference identities and (\ref{eq617}), (\ref{eq618}) implies

\begin{equation}\label{eq626}
\begin{split}
\sin Bt\int_0^t \f{\cos Bs}{B}\, \lambda a^3 P_c\psi^3 ds- \cos Bt\int_0^t \f{\sin Bs}{B}\, \lambda a^3 P_c\psi^3 ds \in L^2_x,
\end{split}
\end{equation}
and
\begin{equation}\label{eq627}
\begin{split}
\sin Bt\int_0^t \f{\sin Bs}{B}\, \lambda a^3 P_c\psi^3 ds+ \cos Bt\int_0^t \f{\cos Bs}{B}\, \lambda a^3 P_c\psi^3 ds \in L^2_x.
\end{split}
\end{equation}
Thus, we conclude
$$\cos Bt\times(\ref{eq626})-\sin Bt\times (\ref{eq627})\in L^2_x,$$
and it is equivalent to

\begin{equation}\label{eq628}
\begin{split}
&\cos Bt\,\sin Bt\int_0^t \f{\cos Bs}{B}\, \lambda a^3 P_c\psi^3 ds- \cos Bt\,\cos Bt\int_0^t \f{\sin Bs}{B}\, \lambda a^3 P_c\psi^3 ds\\
&-\sin Bt\, \sin Bt\int_0^t \f{\sin Bs}{B}\, \lambda a^3 P_c\psi^3 ds- \sin Bt\,\cos Bt\int_0^t \f{\cos Bs}{B}\, \lambda a^3 P_c\psi^3 ds\\
=&-\int_0^t \f{\sin Bs}{B}\, \lambda a^3(s)P_c \psi^3(x)ds \in L^2_x.
\end{split}
\end{equation}
In the same manner, we deduce 
\begin{equation*}
\|\int_t^{+\infty} \cos Bs\, \lambda a^3(s)P_c \psi^3(x)ds\|_{L^2_x}\rightarrow 0,
\end{equation*}

\begin{equation*}
\|\int_t^{+\infty} \sin Bs\, \lambda a^3(s)P_c \psi^3(x)ds\|_{L^2_x}\rightarrow 0,
\end{equation*}

\begin{equation*}
\|\int_t^{+\infty} \f{\sin Bs}{B}\, \lambda a^3(s)P_c \psi^3(x)ds\|_{L^2_x}\rightarrow 0, \, \mbox{as} \, t\rightarrow +\infty.
\end{equation*}
Thus, we finish the proof for Proposition \ref{6.1}.

\subsection{The term $\eta_3(t,x)$} 

We start to estimate $\eta_3(t,x)$. Recall that give two real-valued functions $f(t,x)$ and $g(t,x)$, we denote $$<f,g>:=\int_{\mathbb{R}^3}f(t,x)g(t,x)dx.$$ 
For $\eta_3(t,x)$, we have the following equation 

\begin{equation*}
(\partial^2_t+B^2)\eta_3=\lambda P_c(3a^2\psi^2\eta+3a\psi\eta^2+\eta^3), \quad \eta_3(0,x)=0, \quad \partial_t \eta_3(0,x)=0.
\end{equation*}
This implies

\begin{equation}\label{eta36}
\eta_3(t,x)=\lambda\int_0^t \f{\sin B(t-s)}{B} P_c (3a^2\psi^2\eta+3a\psi\eta^2+\eta^3) ds.
\end{equation}
Here by definition
\begin{equation*}
\begin{split}
&P_c (3a^2\psi^2\eta+3a\psi\eta^2+\eta^3)\\
:=&3a^2\psi^2\eta+3a\psi\eta^2+\eta^3-<3a^2\psi^2\eta+3a\psi\eta^2+\eta^3,\psi>\psi.
\end{split}
\end{equation*}
As in Section 5, we construct an auxiliary function $v(t,x)$ through solving

\begin{equation*}
\begin{split}
(i\partial_t+B)v=&\lambda P_c(3a^2\psi^2\eta+3a\psi\eta^2+\eta^3)\\
v(0,x)=&0.
\end{split}
\end{equation*}
Hence it follows
\begin{equation*}
\begin{split}
v(t,x)=&-i\int_0^t e^{iB(t-s)}\lambda P_c(3a^2\psi^2\eta+3a\psi\eta^2+\eta^3)ds\\
=&-i\int_0^t \sin B(t-s) \lambda P_c(3a^2\psi^2\eta+3a\psi\eta^2+\eta^3)ds\\
&+\int_0^t \cos B(t-s) \lambda P_c(3a^2\psi^2\eta+3a\psi\eta^2+\eta^3)ds.
\end{split}
\end{equation*}
Thus,

\begin{equation}\label{Imv}
\mbox{Im}\, v(t,x)=\int_0^t \sin B(t-s) \lambda P_c(3a^2\psi^2\eta+3a\psi\eta^2+\eta^3)ds,
\end{equation}

\begin{equation}\label{Rev}
\mbox{Re}\, v(t,x)=\int_0^t \cos B(t-s) \lambda P_c(3a^2\psi^2\eta+3a\psi\eta^2+\eta^3)ds.
\end{equation}
Denote
$$\tilde{l}(t):=\|v(t,x)\|^2_{L^2_x}=\|\mbox{Im}\, v(t,x)\|^2_{L^2_x}+\|\mbox{Re}\, v(t,x)\|^2_{L^2_x}.$$

Since $B$ is self-adjoint, a calculation similar to (\ref{dl}) implies
$$\f{d}{dt}\tilde{l}(t)=-2\int_{\mathbb{R}^3}\mbox{Im}\, v(t,x)\lambda P_c(3a^2\psi^2\eta+3a\psi\eta^2+\eta^3)dx.$$
We first prove a proposition for $\tilde{l}(t)$

\begin{proposition}\label{63}
The function $\tilde{l}(t)$ is uniformly bounded for all $t>0$ and has a limit $\tilde{l}(+\infty)$.
\end{proposition}

\begin{proof}
For $\tilde{l}(t)$, we have

\begin{equation*}
\begin{split}
\tilde{l}(t)-\tilde{l}(t_0)=&-2\int_{t_0}^t \int_{\mathbb{R}^3}\mbox{Im}\, v(t',x)\lambda P_c(3a^2\psi^2\eta+3a\psi\eta^2+\eta^3)dxdt'\\
=&-2\int_{t_0}^t \int_{\mathbb{R}^3}\mbox{Im}\, v(t',x)\lambda (3a^2\psi^2\eta+3a\psi\eta^2+\eta^3)dxdt'\\
&+2\int_{t_0}^t \int_{\mathbb{R}^3}\mbox{Im}\, v(t',x)\lambda<3a^2\psi^2\eta+3a\psi\eta^2+\eta^3,\psi>\psi dxdt'.\\
\end{split}
\end{equation*}
We first study the contribution from $\int_{t_0}^t \int_{\mathbb{R}^3}\mbox{Im}\, v(t',x) \eta^3(t',x) dxdt'$. 
Recall that from Proposition 7.6 in \cite{SW}, $\eta$ satisfies
$$\|\eta(t,x)\|_{L^8_x}\leq \f{\d_0}{(1+t)^{\f34}}.$$ 
Let $t_0$ be any fixed number between $0$ and $t$, we derive

\begin{equation}\label{eq638}
\begin{split}
&|\int_{t_0}^t \int_{\mathbb{R}^3}\mbox{Im}\, v(t',x) \eta^3(t',x) dxdt'|\\
=& |\int_{t_0}^t \int_{\mathbb{R}^3}\l \int_0^{t'}\sin B(t'-s) \lambda P_c(3a^2\psi^2\eta+3a\psi\eta^2+\eta^3)(s,x)ds \rr\\
&\times \eta^3(t',x) dxdt'|\\
\leq& \int_{t_0}^t \| \l \int_0^{t'}\sin B(t'-s) \lambda P_c(3a^2\psi^2\eta+3a\psi\eta^2+\eta^3)(s,x)ds \rr \|_{L^2_x}\\
&\times \|\eta^3(t',x)\|_{L^2_x} dt'\\
\leq& \int_{t_0}^t   \int_0^{t'} \| \lambda P_c(3a^2\psi^2\eta+3a\psi\eta^2+\eta^3)(s,x)  \|_{L^2_x}ds \|\eta(t',x)\|_{L^4_x}\|\eta(t',x)\|^2_{L^8_x}dt'\\
\leq& \int_{t_0}^t   \int_0^{t'} \l\f{1}{(1+s)^{\f54}}+\|\eta^3(s,x)\|_{L^2_x}\rr ds \|\eta(t',x)\|_{L^4_x}\|\eta(t',x)\|^2_{L^8_x}dt'\\
\leq& \int_{t_0}^t   \int_0^{t'} \l\f{1}{(1+s)^{\f54}}+\|\eta(s,x)\|_{L^4_x}\|\eta(s,x)\|^2_{L^8_x}\rr ds\\
&\times \|\eta(t',x)\|_{L^4_x}\|\eta(t',x)\|^2_{L^8_x}dt'\\
\leq& \int_{t_0}^t   \int_0^{t'} \f{1}{(1+s)^{\f54}}ds\|\eta(t',x)\|^{\f13}_{L^2_x}\|\eta(t',x)\|^{\f23+2}_{L^8_x}dt'\\
&+ \int_{t_0}^t   \int_0^{t'} \|\eta(s,x)\|^{\f13}_{L^2_x}\|\eta(s,x)\|^{\f23+2}_{L^8_x} ds\|\eta(t',x)\|^{\f13}_{L^2_x}\|\eta(t',x)\|^{\f23+2}_{L^8_x}dt'.
\end{split}
\end{equation}
For the last inequality, we employ
$$\|\eta(t,x)\|_{L^4_x}\leq \|\eta(t,x)\|^{\f13}_{L^2_x}\|\eta(t,x)\|^{\f23}_{L^8_x}.$$
As a consequence of previous subsection, there exists a uniform $C$ such that
$$\|\eta_1(t,x)\|_{L^2_x}+\|\eta_2(t,x)\|_{L^2_x}\leq C.$$
We define $G$ through
\begin{equation*}
G:=\sup_{t>0}\|\eta_3(t,x)\|_{L^2_x}.
\end{equation*}
Together with (\ref{eq638}), we deduce that
\begin{equation*}
\begin{split}
&|\int_{t_0}^t \int_{\mathbb{R}^3}\mbox{Im}\, v(t',x) \eta^3(t',x) dxdt'|\\
\leq& \int_{t_0}^t   \int_0^{t'} \f{1}{(1+s)^{\f54}}ds\|\eta(t',x)\|^{\f13}_{L^2_x}\|\eta(t',x)\|^{\f23+2}_{L^8_x}dt'\\
&+ \int_{t_0}^t   \int_0^{t'} \|\eta(s,x)\|^{\f13}_{L^2_x}\|\eta(s,x)\|^{\f23+2}_{L^8_x} ds\|\eta(t',x)\|^{\f13}_{L^2_x}\|\eta(t',x)\|^{\f23+2}_{L^8_x}dt'\\
\leq&\delta_0^2\int_{t_0}^t \|\eta(t',x)\|^{\f13}_{L^2_x}\cdot \f{1}{(1+t')^2}dt'\\
&+\delta_0^4\int_{t_0}^t \int_0^{t'} \|\eta(s,x)\|^{\f13}_{L^2_x}\cdot \f{1}{(1+s)^2}ds\, \|\eta(t',x)\|^{\f13}_{L^2_x}\cdot \f{1}{(1+t')^2}\\
\leq&\f{\delta_0^2 G^{\f13}}{1+t_0}+\f{\delta_0^4 G^{\f23}}{1+t_0}+\f{\d^2_0 C^{\f13}}{1+t_0}+\f{\delta^4_0 C^{\f23}}{1+t_0}.
\end{split}
\end{equation*}

In the same fashion, for other terms we have
\begin{equation*}
\begin{split}
&|\int_{t_0}^t \int_{\mathbb{R}^3}\mbox{Im}\, v(t',x)\l-\lambda(3a^2\psi^2\eta+3a\psi\eta^2)\\
&\quad \quad \quad \quad \quad \quad \quad \quad \quad +\lambda <3a^2\psi^2\eta+3a\psi\eta^2, \psi> \psi\rr (t',x)\,dxdt'|\\
\leq& |\int_{t_0}^t \int_{\mathbb{R}^3}\l\int_0^{t'}\sin B(t'-s)\lambda P_c(3a^2\psi^2\eta+3a\psi\eta^2+\eta^3)(s,x)ds \rr\\
&\times \l\lambda(3a^2\psi^2\eta+3a\psi\eta^2)+\lambda <3a^2\psi^2\eta+3a\psi\eta^2, \psi> \psi\rr (t',x)dxdt'|\\
\leq& \int_{t_0}^t\int_0^{t'}(\f{1}{(1+s)^{\f54}}+\|\eta^3(s,x)\|_{L^2_x})ds \cdot \f{\delta_0}{(1+t')^{\f54}}dt'\\
\leq& \f{\delta_0}{(1+t_0)^{\f14}}+\int_{t_0}^t\int_0^{t'} \|\eta(s,x)\|^{\f13}_{L^2_x} \|\eta(s,x)\|^{\f23+2}_{L^8_x} ds\, \f{\delta_0}{(1+t')^{\f54}}dt'\\
\leq& \f{\delta_0}{(1+t_0)^{\f14}}+\int_{t_0}^t\int_0^{t'} \|\eta(s,x)\|^{\f13}_{L^2_x}\f{\delta_0^2}{(1+s)^2} ds\, \f{\delta_0}{(1+t')^{\f54}}dt'\\
\leq& \f{\delta_0}{(1+t_0)^{\f14}}+\f{\delta_0^3 G^{\f13}}{(1+t_0)^{\f14}}+\f{\delta^3_0 C^{\f13}}{(1+t_0)^{\f14}}.
\end{split}
\end{equation*}

Putting all the estimates together, for $t\geq t_0$, furthermore, we have also showed that

\begin{equation*}
\begin{split}
&\|\mbox{Im}\,v(t,x)\|^2_{L^2_x}+\|\mbox{Re}\,v(t,x)\|^2_{L^2_x}\\
=&\|v(t,x)\|^2_{L^2_x}\\
\leq&\f{\delta_0^2 G^{\f13}}{1+t_0}+\f{\delta_0^4 G^{\f23}}{1+t_0}+\f{\delta_0}{(1+t_0)^{\f14}}+\f{\delta_0^3 G^{\f13}}{(1+t_0)^{\f14}}\\
&+\f{\delta^2_0 C^{\f13}}{1+t_0}+\f{\delta^4_0 C^{\f23}}{1+t_0}+\f{\delta^3_0 C^{\f13}}{(1+t_0)^{\f14}}. 
\end{split}
\end{equation*}
Recall the definitions from (\ref{Imv}) and (\ref{eta3}):

\begin{equation*}
\mbox{Im}\, v(t,x)=\int_0^t \sin B(t-s) \lambda P_c(3a^2\psi^2\eta+3a\psi\eta^2+\eta^3)ds,
\end{equation*}
and
\begin{equation*}
\eta_3(t,x)=\int_0^t \f{\sin B(t-s)}{B} \lambda P_c(3a^2\psi^2\eta+3a\psi\eta^2+\eta^3)ds.
\end{equation*}
We first pick up $t_0=0$ (initial time to be $0$) and deduce

\begin{equation*}
\begin{split}
&\|\eta_3(t,x)\|^2_{L^2_x}\\
\leq&\|\mbox{Im}\,v(t,x)\|^2_{L^2_x}\leq\|v(t,x)\|^2_{L^2_x}\\
\leq&{\delta_0^2 G^{\f13}}+{\delta_0^4 G^{\f23}}+{\delta_0}+{\delta_0^3 G^{\f13}}+\d^2_0 C^{\f13}+\d^4_0 C^{\f23}+\d_0^3 C^{\f13}.
\end{split}
\end{equation*}
Since $G$ is defined as $\sup_{t>0} \|\eta_3(t,x)\|_{L^2_x}$, we arrive at

$$G^2\leq{\delta_0^2 G^{\f13}}+{\delta_0^4 G^{\f23}}+{\delta_0}+{\delta_0^3 G^{\f13}}+\d^2_0 C^{\f13}+\d^4_0 C^{\f23}+\d_0^3 C^{\f13}.$$
This implies

$$G:=\sup_{t>0} \|\eta_3(t,x)\|_{L^2_x}\leq \d_0^{\f12}+\d_0 C^{\f16}.$$
If our initial time starts from an arbitrary large $t_0>0$, with the same argument we derive

$$\sup_{t>t_0} \|\eta_3(t,x)\|_{L^2_x}\leq \f{\d_0^{\f12}+\d_0 C^{\f16}}{(1+t_0)^{\f18}},$$
which goes to $0$ as $t_0$ goes to $+\infty$. Thus, let $t_0=0$, we have showed
$$\tilde{l}(t)\quad \mbox{is uniformly bounded for any}\quad t>0.$$
Set $t_0>0$ to be a large number, we have also proved
$$\tilde{l}(t)-\tilde{l}(t_0)\rightarrow 0 \quad \mbox{for any} \quad t\geq t_0 \quad \mbox{and} \quad t_0\rightarrow +\infty.$$
This finishes the proof for Proposition \ref{63}.

\end{proof}

Now we are ready to prove

\begin{proposition}\label{6.4}
For $a(t)$, $\psi(t,x)$ and $\eta(t,x)$ as in Sections 1-5, we have
$$\int_0^{+\infty} \cos Bs \,\lambda P_c(3a^2\psi^2\eta+3a\psi\eta^2+\eta^3)(s,x) ds \in L^2_x,$$
$$\int_0^{+\infty} \sin Bs \,\lambda P_c(3a^2\psi^2\eta+3a\psi\eta^2+\eta^3)(s,x) ds \in L^2_x,$$
$$\int_0^{+\infty} \f{\sin Bs}{B} \,\lambda P_c(3a^2\psi^2\eta+3a\psi\eta^2+\eta^3)(s,x) ds \in L^2_x,$$
and
$$\|\int_t^{+\infty} \cos Bs \,\lambda P_c(3a^2\psi^2\eta+3a\psi\eta^2+\eta^3) ds\|_{L^2_x}\rightarrow 0,$$
$$\|\int_t^{+\infty} \sin Bs \,\lambda P_c(3a^2\psi^2\eta+3a\psi\eta^2+\eta^3) ds\|_{L^2_x}\rightarrow 0,$$
$$\|\int_t^{+\infty} \f{\sin Bs}{B} \,\lambda P_c(3a^2\psi^2\eta+3a\psi\eta^2+\eta^3) ds\|_{L^2_x}\rightarrow 0,$$
as $t\rightarrow +\infty$.

\end{proposition}

\begin{proof}
Together with the conclusion in Proposition \ref{63}, these can be proved by similar arguments as for Proposition \ref{6.1}. 
\end{proof}

Since $\sin Bt / B$ and $\cos Bt$ are bounded operators for $L^2_x$, Proposition \ref{6.4} implies
\begin{equation}\label{new 6.65}
\begin{split}
\|-\lambda \f{\sin Bt}{B}\int_t^{+\infty} \cos Bs P_c(3a^2\psi^2\eta+3a\psi\eta^2+\eta^3)(s,x) ds\|_{L^2_x}&\rightarrow 0,\\
\|\lambda \cos Bt \int_t^{+\infty} \f{\sin Bs}{B} P_c(3a^2\psi^2\eta+3a\psi\eta^2+\eta^3)(s,x) ds\|_{L^2_x}&\rightarrow 0,
\end{split}
\end{equation}
as $t\rightarrow +\infty$. Furthermore, from Proposition \ref{6.4} we also have as $t\rightarrow +\infty$
\begin{equation}\label{new 6.66}
\begin{split}
&\|B\cdot \l -\lambda \f{\sin Bt}{B}\int_t^{+\infty} \cos Bs P_c(3a^2\psi^2\eta+3a\psi\eta^2+\eta^3)(s,x) ds \rr\|_{L^2_x}\\
=&\|-\lambda \sin Bt \int_t^{+\infty} \cos Bs P_c(3a^2\psi^2\eta+3a\psi\eta^2+\eta^3)(s,x) ds \|_{L^2_x}\rightarrow 0,\\
\end{split}
\end{equation}

\begin{equation}\label{new 6.67}
\begin{split}
&\|B\cdot \l \lambda \cos Bt \int_t^{+\infty} \f{\sin Bs}{B} P_c(3a^2\psi^2\eta+3a\psi\eta^2+\eta^3)(s,x) ds \rr\|_{L^2_x}\\
=& \| \lambda \cos Bt \int_t^{+\infty} \sin Bs P_c(3a^2\psi^2\eta+3a\psi\eta^2+\eta^3)(s,x) ds \|_{L^2_x}\rightarrow 0.
\end{split}
\end{equation}
Recall (\ref{B lemma}): for any $f\in H^1_x$, we have
$$\|f\|^2_{\dot{H}^!_x}\leq \|Bf\|^2_{L^2_x}+C\|f\|^2_{L^2_x}.$$ 
Hence, (\ref{new 6.65}), (\ref{new 6.66}) and (\ref{new 6.67}) imply

\begin{equation*}
\begin{split}
\|-\lambda \f{\sin Bt}{B}\int_t^{+\infty} \cos Bs P_c(3a^2\psi^2\eta+3a\psi\eta^2+\eta^3)(s,x) ds\|_{H^1_x}&\rightarrow 0,\\
\|\lambda \cos Bt \int_t^{+\infty} \f{\sin Bs}{B} P_c(3a^2\psi^2\eta+3a\psi\eta^2+\eta^3)(s,x) ds\|_{H^1_x}&\rightarrow 0,
\end{split}
\end{equation*}
as $t\rightarrow +\infty$.
Therefore, we arrive at
\begin{proposition}\label{6.5}
For $\eta_3$ is define through (\ref{eta3}), we have
$$\eta_3(t,x)=\mbox{free wave}+R_3(t,x), \,\mbox{where}\, \lim_{t\rightarrow +\infty} \|R_3(t,x)\|_{H^1_x}=0.$$
\end{proposition}
Gathering all the conclusions together, Theorem \ref{thm1.3} and Theorem \ref{thm1.4} follow. \\

Furthermore, we also achieve an explicit expression for scattering states $S_1(x)$ and $S_2(x)$:

\begin{theorem}\label{scattering states}(Scattering States)
For solution $u(t,x)$ to equation 
\begin{equation*}
\partial^2_{t}u-\Delta u+V(x) u+m^2 u=\lambda u^3, \quad \lambda \in \mathbb{R}\backslash 0.
\end{equation*}
\begin{equation*}
u(0,x)=u_0(x), \quad \partial_{t}u(0,x)=u_1(x).
\end{equation*}
as $t\rightarrow +\infty$, we have
$$u(t,x)=2\rho(t)\cos\theta(t)\psi(x)+\f{\sin Bt}{B}S_1(x)+\cos Bt S_2(x)+R(t,x), \,\, \mbox{where}$$
\begin{equation*}
\begin{split}
S_1(x)=&P_c u_1+\lambda\int_0^{+\infty}\cos Bs a^3(s) P_c\psi^3 ds\\
&\quad-\lambda \int_0^{+\infty} \cos Bs P_c(3a^2\psi^2\eta+3a\psi\eta^2+\eta^3)(s,x) ds \in L^2_x.
\end{split}
\end{equation*}
\begin{equation*}
\begin{split}
S_2(x)=&P_c u_0-\lambda\int_0^{+\infty}\f{\sin Bs}{B} a^3(s) P_c\psi^3 ds\\
&\quad+\lambda \int_t^{+\infty} \f{\sin Bs}{B} P_c(3a^2\psi^2\eta+3a\psi\eta^2+\eta^3)(s,x) ds \in H^1_x,
\end{split}
\end{equation*}
$$\rho(t)\approx\f{\rho(0)}{(1+t)^{\f14}}, \quad \theta(t)-\O t=\mathcal{O}(t^{\f12}), \quad \mbox{and} \quad \lim_{t\rightarrow +\infty} \|R(t,x)\|_{H^1_x}=0.$$
\end{theorem}

\appendix

\section{Estimates on $E'(t)$}

The goal of these two sections in appendix is to prove

$$|E'(t)|\leq \f{\d_0^2}{1+t}.$$

\subsection{Expression for $E'(t)$}

From (\ref{2.17}) in Section 2, we have

\begin{equation*}
\begin{split}
  A'(t)=&(2i\O)^{-1}e^{-i\O t}F(a,\eta)\\
 =&(2i\O)^{-1}e^{-i\O t}\cdot\lambda\cdot[a^3(t)\int \psi^4+3a^2(t)\int \psi^3\eta+3a(t)\int \psi^2\eta^2+\int\psi\eta^3]
 \end{split}
 \end{equation*}

One the other {\color{black}hand},  Proposition \ref{3.1} gives 

\begin{equation*}
\begin{split}
&A'(t)\\
=&\frac{-i\lambda}{2\Omega}\|\psi\|^4_4(A^3e^{2i\Omega t}+3|A|^2A+3|A|^2\bar{A}e^{-2i\Omega t}+\bar{A}^3e^{-4i\Omega t})\\
&-\frac{3\lambda^2}{4\Omega}\Gamma|A|^4A+\frac{3i\lambda^2}{4\Omega}[\Lambda-5\underline{\rho}(\Omega)+3\underline{\rho}(-\Omega)+\underline{\rho}(-3\Omega)]|A|^4A\\
&-\f{3\lambda^2}{4\Omega}A^5 e^{4i\O t}[i\Lambda+\Gamma-\underline{\rho}(-3\Omega)]\\
&-\f{3i\lambda^2}{4\Omega}|A|^2A^3e^{2i\Omega t}[-2\Lambda+2i\Gamma+\underline{\rho}(\Omega)+i\underline{\rho}(-3\Omega)+3\underline{\rho}(-\Omega)]\\
&-\f{3i\lambda^2}{4\Omega}|A|^4\bar{A}e^{-2i\Omega t}[7\underline{\rho}(\Omega)+9\underline{\rho}(-\Omega)+\underline{\rho}(-3\Omega)+\underline{\rho}(3\Omega+i0)]\\
&-\f{3i\lambda^2}{4\Omega}|A|^2\bar{A}^3e^{-4i\Omega t}[3\underline{\rho}(-\Omega)-2i\underline{\rho}(-3\Omega)+\underline{\rho}(3\Omega+i0)+3\underline{\rho}(\Omega)]\\
&-\f{3i\lambda^2}{4\Omega}\bar{A}^5e^{-6i\Omega t}[i\underline{\rho}(-3\Omega)-\frac43 i\underline{\rho}(3\Omega+i0)]+E,
\end{split}
\end{equation*}

These imply

\begin{equation*}
\begin{split}
E(t)=&(2i\O)^{-1}e^{-i\O t}\cdot\lambda\cdot[a^3(t)\int \psi^4+3a^2(t)\int \psi^3\eta+3a(t)\int \psi^2\eta^2+\int\psi\eta^3]\\
&+\frac{i\lambda}{2\Omega}\|\psi\|^4_4(A^3e^{2i\Omega t}+3|A|^2A+3|A|^2\bar{A}e^{-2i\Omega t}+\bar{A}^3e^{-4i\Omega t})\\
&+\frac{3\lambda^2}{4\Omega}\Gamma|A|^4A+\frac{3i\lambda^2}{4\Omega}[\Lambda-5\underline{\rho}(\Omega)+3\underline{\rho}(-\Omega)+\underline{\rho}(-3\Omega)]|A|^4A\\
&+\f{3\lambda^2}{4\Omega}A^5 e^{4i\O t}[i\Lambda+\Gamma-\underline{\rho}(-3\Omega)]\\
&+\f{3i\lambda^2}{4\Omega}|A|^2A^3e^{2i\Omega t}[-2\Lambda+2i\Gamma+\underline{\rho}(\Omega)+i\underline{\rho}(-3\Omega)+3\underline{\rho}(-\Omega)]\\
&+\f{3i\lambda^2}{4\Omega}|A|^4\bar{A}e^{-2i\Omega t}[7\underline{\rho}(\Omega)+9\underline{\rho}(-\Omega)+\underline{\rho}(-3\Omega)+\underline{\rho}(3\Omega+i0)]\\
&+\f{3i\lambda^2}{4\Omega}|A|^2\bar{A}^3e^{-4i\Omega t}[3\underline{\rho}(-\Omega)-2i\underline{\rho}(-3\Omega)+\underline{\rho}(3\Omega+i0)+3\underline{\rho}(\Omega)]\\
&+\f{3i\lambda^2}{4\Omega}\bar{A}^5e^{-6i\Omega t}[i\underline{\rho}(-3\Omega)-\frac43 i\underline{\rho}(3\Omega+i0)].
\end{split}
\end{equation*}

Since $a(t)=Ae^{i\O t}+\bar{A}e^{-i\O t}$, it is straightforward to check

$$(2i\O)^{-1}e^{-i\O t}\cdot\lambda\cdot a^3(t)\int \psi^4=\frac{-i\lambda}{2\Omega}\|\psi\|^4_4(A^3e^{2i\Omega t}+3|A|^2A+3|A|^2\bar{A}e^{-2i\Omega t}+\bar{A}^3e^{-4i\Omega t}).$$

Therefore, we arrive at

\begin{proposition}\label{a2}
\begin{equation*}
\begin{split}
E(t)=&(2i\O)^{-1}e^{-i\O t}\cdot\lambda\cdot[3a^2(t)\int \psi^3\eta+3a(t)\int \psi^2\eta^2+\int\psi\eta^3]\\
&+\frac{3\lambda^2}{4\Omega}\Gamma|A|^4A+\frac{3i\lambda^2}{4\Omega}[\Lambda-5\underline{\rho}(\Omega)+3\underline{\rho}(-\Omega)+\underline{\rho}(-3\Omega)]|A|^4A\\
&+\f{3\lambda^2}{4\Omega}A^5 e^{4i\O t}[i\Lambda+\Gamma-\underline{\rho}(-3\Omega)]\\
&+\f{3i\lambda^2}{4\Omega}|A|^2A^3e^{2i\Omega t}[-2\Lambda+2i\Gamma+\underline{\rho}(\Omega)+i\underline{\rho}(-3\Omega)+3\underline{\rho}(-\Omega)]\\
&+\f{3i\lambda^2}{4\Omega}|A|^4\bar{A}e^{-2i\Omega t}[7\underline{\rho}(\Omega)+9\underline{\rho}(-\Omega)+\underline{\rho}(-3\Omega)+\underline{\rho}(3\Omega+i0)]\\
&+\f{3i\lambda^2}{4\Omega}|A|^2\bar{A}^3e^{-4i\Omega t}[3\underline{\rho}(-\Omega)-2i\underline{\rho}(-3\Omega)+\underline{\rho}(3\Omega+i0)+3\underline{\rho}(\Omega)]\\
&+\f{3i\lambda^2}{4\Omega}\bar{A}^5e^{-6i\Omega t}[i\underline{\rho}(-3\Omega)-\frac43 i\underline{\rho}(3\Omega+i0)].
\end{split}
\end{equation*}
\end{proposition}

\begin{remark}
Since
$$|a(t)|\leq \f{\d_0}{(1+t)^{\f14}}, \quad \|\eta\|_{L^8_x}\leq \f{\d_0}{(1+t)^{\f34}},$$
we have
$$|E(t)|\leq \f{\d^2_0}{(1+t)^{\f54}}.$$
\end{remark}

We then take {\color{black}a} derivative of $t$ on both sides of {\color{black}the} above proposition

\begin{proposition}
\begin{equation*}
\begin{split}
E'(t)=&(-i\O)\cdot(2i\O)^{-1}e^{-i\O t}\cdot\lambda\cdot[3a^2(t)\int \psi^3\eta+3a(t)\int \psi^2\eta^2+\int\psi\eta^3]\\
&+(2i\O)^{-1}e^{-i\O t}\cdot\lambda\cdot[6a(t)a'(t)\int \psi^3\eta+3a'(t)\int \psi^2\eta^2+\int\psi\eta^3]\\
&+(2i\O)^{-1}e^{-i\O t}\cdot\lambda\cdot[3a^2(t)\int \psi^3\partial_t\eta+6a(t)\int \psi^2\eta\cdot\partial_t\eta+3\int\psi\eta^2\cdot\partial_t\eta]\\
&+\bigg(\frac{3\lambda^2}{4\Omega}\Gamma|A|^4A+\frac{3i\lambda^2}{4\Omega}[\Lambda-5\underline{\rho}(\Omega)+3\underline{\rho}(-\Omega)+\underline{\rho}(-3\Omega)]|A|^4A\bigg)'\\
&+\f{3\lambda^2}{4\Omega}\bigg(A^5 e^{4i\O t}[i\Lambda+\Gamma-\underline{\rho}(-3\Omega)]\bigg)'\\
&+\f{3i\lambda^2}{4\Omega}\bigg(|A|^2A^3e^{2i\Omega t}[-2\Lambda+2i\Gamma+\underline{\rho}(\Omega)+i\underline{\rho}(-3\Omega)+3\underline{\rho}(-\Omega)]\bigg)'\\
&+\f{3i\lambda^2}{4\Omega}\bigg(|A|^4\bar{A}e^{-2i\Omega t}[7\underline{\rho}(\Omega)+9\underline{\rho}(-\Omega)+\underline{\rho}(-3\Omega)+\underline{\rho}(3\Omega+i0)]\bigg)'\\
&+\f{3i\lambda^2}{4\Omega}\bigg(|A|^2\bar{A}^3e^{-4i\Omega t}[3\underline{\rho}(-\Omega)-2i\underline{\rho}(-3\Omega)+\underline{\rho}(3\Omega+i0)+3\underline{\rho}(\Omega)]\bigg)'\\
&+\f{3i\lambda^2}{4\Omega}\bigg(\bar{A}^5e^{-6i\Omega t}[i\underline{\rho}(-3\Omega)-\frac43 i\underline{\rho}(3\Omega+i0)]\bigg)'.
\end{split}
\end{equation*}
\end{proposition}

We then move to prove

\begin{proposition}
\begin{equation}\label{A1}
|E'(t)|\leq \f{\d^2_0}{1+t}.
\end{equation}
\end{proposition}

\begin{proof}

Since $$|a(t)|+|a'(t)|+|A'(t)|+|\bar{A}'|\leq \f{\d_0}{(1+t)^{\f14}}, \quad \|\eta\|_{L^8_x}\leq \f{\d_0}{(1+t)^{\f34}},$$

we have

\begin{equation*}
\begin{split}
&|(-i\O)\cdot(2i\O)^{-1}e^{-i\O t}\cdot\lambda\cdot[3a^2(t)\int \psi^3\eta+3a(t)\int \psi^2\eta^2+\int\psi\eta^3]|\\
&+|(2i\O)^{-1}e^{-i\O t}\cdot\lambda\cdot[6a(t)a'(t)\int \psi^3\eta+3a'(t)\int \psi^2\eta^2+\int\psi\eta^3]|\\
\lesssim& \f{\d_0^3}{(1+t)^{\f54}}.
\end{split}
\end{equation*}

For terms of the form $e^{im\O t}A^{5-k}(t)\bar{A}^k(t)$, where $m$ and $k$ are integers and $0\leq k \leq 5$, we have

$$|\bigg(e^{im\O t}A^{5-k}\bar{A}^k\bigg)'|\leq |m\O\cdot A^{5-k}\bar{A}^k|+|(5-k)A^{4-k}\cdot A'\cdot \bar{A}^k|+|k A^{5-k}\bar{A}^{k-1}\bar{A}'|\lesssim \f{\d_0^5}{(1+t)^{\f54}}.$$

{\bf Claim:} In {\color{black}the} next section, we will prove

\begin{equation}\label{A2}
\|\partial_t \eta\|_{L^8_x}\leq \f{\d_0}{(1+t)^{\f12}}.
\end{equation}

Under this claim, we have

$$|(2i\O)^{-1}e^{-i\O t}\cdot\lambda\cdot[3a^2(t)\int \psi^3\partial_t\eta+6a(t)\int \psi^2\eta\cdot\partial_t\eta+3\int\psi\eta^2\cdot\partial_t\eta]|\leq \f{\d_0^3}{1+t}.$$

Gathering these estimates, we then arrive at (A.1) under the condition of (A.2). We then move to estimate $\partial_t \eta$.

\end{proof}

\section{Estimates on $\partial_t \eta$}

From (\ref{2.6}), $\eta$ satisfies

\begin{equation*}
\partial^2_t \eta+B^2\eta=\lambda P_c(a\psi+\eta)^3=\lambda P_c (a^3\psi^3+3a^2\psi^2\eta+3a\psi\eta^2+\eta^3).
\end{equation*}

$$\eta(x,0)=P_c u_0, \quad \partial_t\eta(x,0)=P_c u_1.$$

We hence have

\begin{equation*}
\begin{split}
&\partial^2_t (\partial_t\eta)+B^2(\partial_t\eta)\\
=&\lambda \bigg( P_c (a^3\psi^3+3a^2\psi^2\eta+3a\psi\eta^2+\eta^3)\bigg)'\\
=&\lambda P_c (3a^2\cdot a'\cdot\psi^3+6a\cdot a'\cdot\psi^2\cdot\eta+3a'\cdot\psi\cdot\eta^2)\\
&+\lambda P_c (3a^2\cdot\psi^2\cdot\partial_t\eta+6a\cdot\psi\cdot\eta\cdot\partial_t\eta+3\eta^2\cdot\partial_t\eta).
\end{split}
\end{equation*}

For initial data, we have

$$(\partial_t\eta)(x,0)=P_c u_1,$$

and

\begin{equation*}
\begin{split}
\partial_t (\partial_t \eta)(x,0)=&-B^2\eta(0,x)+\lambda P_c(a\psi+\eta)^3(x,0)\\
=&-B^2 P_c u_0+\lambda P_c (<\psi, u_0>\psi+P_c u_0)^3,
\end{split}
\end{equation*}

where we use $a(0)=<\psi,u_0>$.

By standard dispersive estimates (see (7.19), (7.20) and Proposition 7.2 in \cite{SW}), we have
\begin{equation*}
\begin{split}
\|\partial_t \eta(t,x)\|_{L^8_x}\leq&|\lambda|\int_{0}^t\|E_1(t-s)P_c(3a^2\cdot a'\cdot\psi^3+6a\cdot a'\cdot\psi^2\cdot\eta+3a'\cdot\psi\cdot\eta^2)\|_{L^8_x}ds\\
&+|\lambda|\int_{0}^t\|E_1(t-s)P_c(3a^2\cdot\psi^2\cdot\partial_t\eta+6a\cdot\psi\cdot\eta\cdot\partial_t\eta+3\eta^2\cdot\partial_t\eta)\|_{L^8_x}ds\\
\leq&|\lambda|\int_{0}^{t-1} |t-s|^{-\f98}\|P_c(3a^2\cdot a'\cdot\psi^3+6a\cdot a'\cdot\psi^2\cdot\eta+3a'\cdot\psi\cdot\eta^2)\|_{W^{1,\f87}_x}ds\\
&+|\lambda|\int_{t-1}^t |t-s|^{-\f38}\|P_c(3a^2\cdot a'\cdot\psi^3+6a\cdot a'\cdot\psi^2\cdot\eta+3a'\cdot\psi\cdot\eta^2)\|_{W^{1,\f87}_x}ds\\
&+|\lambda|\int_{0}^{t-1} |t-s|^{-\f98}\|P_c(3a^2\cdot\psi^2\cdot\partial_t\eta+6a\cdot\psi\cdot\eta\cdot\partial_t\eta+3\eta^2\cdot\partial_t\eta)\|_{W^{1,\f87}_x}ds\\
&+|\lambda|\int_{t-1}^{t} |t-s|^{-\f38}\|P_c(3a^2\cdot\psi^2\cdot\partial_t\eta+6a\cdot\psi\cdot\eta\cdot\partial_t\eta+3\eta^2\cdot\partial_t\eta)\|_{W^{1,\f87}_x}ds\\
=&V_1+V_2+V_3+V_4. 
\end{split}
\end{equation*}
For $V_1$, we use similar H\"older's inequalities as in Proposition 7.2 in \cite{SW} and derive
\begin{equation*}
\begin{split}
V_1=&|\lambda|\int_{0}^{t-1} |t-s|^{-\f98}\|P_c(3a^2\cdot a'\cdot\psi^3+6a\cdot a'\cdot\psi^2\cdot\eta+3a'\cdot\psi\cdot\eta^2)\|_{W^{1,\f87}_x}ds\\
\leq&|\lambda|C_{\psi}\int_0^{t-1}|t-s|^{-\f98}\bigg((1+s)^{-\f34}\d_0^3+(1+s)^{-\f12}\d_0^2(\|\eta\|_{L^8_x}+\|B\eta\|_{L^4_x})\bigg)\\
&\quad\quad\quad\quad\quad\quad\quad\quad\quad\quad\quad\quad\quad\quad\quad +(1+s)^{-\f14}\d_0\bigg(\|\eta\|^2_{L^8_x}+\|B\eta\|_{L^4_x}\|\eta\|_{L^8_x}\bigg)ds\\
\leq&\f{|\lambda|C_{\psi}\cdot\d_0^3}{(1+t)^{\f78}}.
\end{split}
\end{equation*}
We use $C_{\psi}$ to represent a constant depending on $\psi$. And for the last step, we employ the {\color{black}following} estimates in \cite{SW} (see (7.33), (7.35), (7.40) in \cite{SW})
$$\|\eta(s,x)\|_{L^8_x}\leq \f{\d_0}{(1+s)^{\f34}}, \quad \|B\eta(s,x)\|_{L^4_x}\leq \f{\d_0}{(1+s)^{\f14}}.$$
For $V_2$, it holds
\begin{equation*}
\begin{split}
V_2=&|\lambda|\int_{t-1}^t |t-s|^{-\f38}\|P_c(3a^2\cdot a'\cdot\psi^3+6a\cdot a'\cdot\psi^2\cdot\eta+3a'\cdot\psi\cdot\eta^2)\|_{W^{1,\f87}_x}ds\\
\leq&|\lambda|C_{\psi}\int_{t-1}^{t}|t-s|^{-\f38}(s^{-\f34}\d_0^3+s^{-\f12}\d_0^2(\|\eta\|_{L^8_x}+\|B\eta\|_{L^4_x})\\
&+s^{-\f14}\d_0(\|\eta\|^2_{L^8_x}+\|B\eta\|_{L^4_x}\|\eta\|_{L^8_x})ds\\
\leq&\f{|\lambda|C_{\psi}\cdot\d_0^3}{(1+t)^{\f34}}.
\end{split}
\end{equation*}
We then move to estimate $V_3$. By the definition of $P_c$, for any function $f(x)\in L^p_x,\, 1<p\leq+\infty$,  we have
$$\|P_c f\|_{L^p_x}=\|f-<f,\psi>\psi\|_{L^p_x}\leq C_{\psi}\|f\|_{L^p_x}.$$
$$\|\partial P_c f\|_{L^p_x}=\|\partial f-<f,\psi>\partial\psi\|_{L^p_x}\leq C_{\psi}(\|f\|_{L^p_x}+\|\partial f\|_{L^p_x}).$$
Hence, we derive
\begin{equation*}
\begin{split}
V_3=&|\lambda|\int_{0}^{t-1} |t-s|^{-\f98}\|P_c(3a^2\cdot\psi^2\cdot\partial_t\eta+6a\cdot\psi\cdot\eta\cdot\partial_t\eta+3\eta^2\cdot\partial_t\eta)\|_{L^{\f87}_x}ds\\
&+|\lambda|\int_{0}^{t-1} |t-s|^{-\f98}\|\partial P_c(3a^2\cdot\psi^2\cdot\partial_t\eta+6a\cdot\psi\cdot\eta\cdot\partial_t\eta+3\eta^2\cdot\partial_t\eta)\|_{L^{\f87}_x}ds\\
\leq&|\lambda| C_{\psi} \int_{0}^{t-1} |t-s|^{-\f98}\|(3a^2\cdot\psi^2\cdot\partial_t\eta+6a\cdot\psi\cdot\eta\cdot\partial_t\eta+3\eta^2\cdot\partial_t\eta)\|_{L^{\f87}_x}ds\\
&+|\lambda| C_{\psi}\int_{0}^{t-1} |t-s|^{-\f98}\|\partial (3a^2\cdot\psi^2\cdot\partial_t\eta+6a\cdot\psi\cdot\eta\cdot\partial_t\eta+3\eta^2\cdot\partial_t\eta)\|_{L^{\f87}_x}ds\\
=&V_{31}+V_{32}.
\end{split}
\end{equation*}
Together with H\"older's inequality and the fact $\|\eta\|_{L^2_x}+\|\partial_t \eta\|_{L^2_x}\leq \d_0$ (coming from energy estimate in \cite{SW}), we bound $V_{31}$:
\begin{equation*}
\begin{split}
V_{31}=&|\lambda|C_{\psi}\int_{0}^{t-1} |t-s|^{-\f98}\|(3a^2\cdot\psi^2\cdot\partial_t\eta+6a\cdot\psi\cdot\eta\cdot\partial_t\eta+3\eta^2\cdot\partial_t\eta)\|_{L^{\f87}_x}ds\\
\leq&|\lambda|C_{\psi}\int_{0}^{t-1} |t-s|^{-\f98}\bigg((1+s)^{-\f12}\d_0^2\cdot \|\partial_t\eta(s,x)\|_{L^2_x}\\
+&(1+s)^{-\f14}\d_0\cdot\|\eta(s,x)\|_{L^8_x}\cdot \|\partial_t\eta(s,x)\|_{L^2_x}+\|\eta(s,x)\|^{\f13}_{L^2_x}\|\eta(s,x)\|^{\f53}_{L^8_x}\|\partial_t\eta(s,x)\|_{L^2_x}\bigg)ds\\
\leq&\f{|\lambda |C_{\psi}\cdot\d_0^3}{(1+t)^{\f58}}.
\end{split}
\end{equation*}

For $V_{32}$, we have
\begin{equation*}
\begin{split}
V_{32}
=&|\lambda|C_{\psi}\int_{0}^{t-1} |t-s|^{-\f98}\|\partial(3a^2\cdot\psi^2\cdot\partial_t\eta+6a\cdot\psi\cdot\eta\cdot\partial_t\eta+3\eta^2\cdot\partial_t\eta)\|_{L^{\f87}_x}ds\\
\leq&|\lambda|C_{\psi}\int_{0}^{t-1} |t-s|^{-\f98}\bigg((1+s)^{-\f12}\d_0^2\cdot \|\partial_t\eta(s,x)\|_{L^2_x}+(1+s)^{-\f12}\d_0^2\cdot \|\partial\partial_t\eta(s,x)\|_{L^2_x}\\
+&(1+s)^{-\f14}\d_0\cdot\|\eta(s,x)\|_{L^8_x}\cdot \|\partial\partial_t\eta(s,x)\|_{L^2_x}+(1+s)^{-\f14}\d_0\cdot\|\partial\eta(s,x)\|_{L^4_x}\cdot \|\partial_t\eta(s,x)\|_{L^2_x}\\
+&\|\eta(s,x)\|^{\f13}_{L^2_x}\|\eta(s,x)\|^{\f53}_{L^8_x}\|\partial\partial_t\eta(s,x)\|_{L^2_x}+\|\partial\eta(s,x)\|^{\f13}_{L^2_x}\|\eta(s,x)\|^{\f53}_{L^8_x}\|\partial_t\eta(s,x)\|_{L^2_x}\bigg)ds\\
\leq&\f{|\lambda |C_{\psi}\cdot\d_0^3}{(1+t)^{\f58}}.
\end{split}
\end{equation*}
For the last inequality, we use the $L^p$ boundedness of $\partial B^{-1}$ for $p\geq 1$ (see also the proof of (7.24) in \cite{SW}): 
$$\|\partial\eta(s,x)\|_{L^4_x}\leq \|B\eta(s,x)\|_{L^4_x}, \quad \|\partial\partial_t\eta\|_{L^2_x}\leq \|B\partial_t\eta\|_{L^2_x},$$
and we employ the estimate
$$\|B\eta(s,x)\|_{L^4_x}\leq\f{\d_0}{(1+s)^{\f14}}, \,\mbox{and energy estimate}\,\|B\partial_t\eta\|_{L^2_x}\lesssim \d_0.$$

Thus, we obtain
$$V_3\leq\f{|\lambda |C_{\psi}\cdot\d_0^2}{(1+t)^{\f58}}$$
We treat $V_4$ in the same fashion as for $V_3$  and obtain
\begin{equation*}
\begin{split}
V_{4}&\leq|\lambda|C_{\psi}\int_{t-1}^{t} |t-s|^{-\f38}\|(3a^2\cdot\psi^2\cdot\partial_t\eta+6a\cdot\psi\cdot\eta\cdot\partial_t\eta+3\eta^2\cdot\partial_t\eta)\|_{L^{\f87}_x}ds\\
&+|\lambda|C_{\psi}\int_{t-1}^{t} |t-s|^{-\f38}\|\partial(3a^2\cdot\psi^2\cdot\partial_t\eta+6a\cdot\psi\cdot\eta\cdot\partial_t\eta+3\eta^2\cdot\partial_t\eta)\|_{L^{\f87}_x}ds\\
\leq&|\lambda|C_{\psi}\int_{t-1}^{t} |t-s|^{-\f38}\bigg((1+s)^{-\f12}\d_0^2\cdot \|\partial_t\eta(s,x)\|_{L^2_x}\\
+&(1+s)^{-\f14}\d_0\cdot\|\eta(s,x)\|_{L^8_x}\cdot \|\partial_t\eta(s,x)\|_{L^2_x}+\|\eta(s,x)\|^{\f13}_{L^2_x}\|\eta(s,x)\|^{\f53}_{L^8_x}\|\partial_t\eta(s,x)\|_{L^2_x}\bigg)ds\\
+&|\lambda|C_{\psi}\int_{t-1}^{t} |t-s|^{-\f38}\bigg((1+s)^{-\f12}\d_0^2\cdot \|\partial_t\eta(s,x)\|_{L^2_x}+(1+s)^{-\f12}\d_0^2\cdot \|\partial\partial_t\eta(s,x)\|_{L^2_x}\\
+&(1+s)^{-\f14}\d_0\cdot\|\eta(s,x)\|_{L^8_x}\cdot \|\partial\partial_t\eta(s,x)\|_{L^2_x}+(1+s)^{-\f14}\d_0\cdot\|\partial\eta(s,x)\|_{L^4_x}\cdot \|\partial_t\eta(s,x)\|_{L^2_x}\\
+&\|\eta(s,x)\|^{\f13}_{L^2_x}\|\eta(s,x)\|^{\f53}_{L^8_x}\|\partial\partial_t\eta(s,x)\|_{L^2_x}+\|\partial\eta(s,x)\|^{\f13}_{L^2_x}\|\eta(s,x)\|^{\f53}_{L^8_x}\|\partial_t\eta(s,x)\|_{L^2_x}\bigg)ds\\
\leq&\f{|\lambda |C_{\psi}\cdot\d_0^3}{(1+t)^{\f12}}.
\end{split}
\end{equation*}

Putting all the estimates together for $V_1, V_2, V_3, V_4$, we hence {\color{black}proved}
$$\|\partial_t\eta(t,x)\|_{L^8_x}\leq \f{\d_0}{(1+t)^{\f12}}.$$


\begin{thebibliography}{99}

\bibitem[1] {BC11} \newblock D. Bambusi, S. Cuccagna, \newblock {On dispersion of small energy solutions to the nonlinear Klein-Gordon equation with a potential,} \newblock \emph{Amer. J. Math.}, \textbf{133} (2011), no.5, 1421--1468.

\bibitem[2] {B-Pe}  \newblock V.-S. Buslaev, G.-S. Perelman, \newblock {On the stability of solitary waves for nonlinear Schr\"odinger equations,} \newblock \emph{Amer. Math. Soc. Transl. Ser.} 2, Vol. \textbf{164} (1995), 75-98.

\bibitem[3]{B-S} \newblock V.-S. Buslaev, C. Sulem, \newblock {On asymptotic stability of solitary waves for nonlinear Schr\"odinger equations,} \newblock \emph{Ann. Inst. H. Poincar\'e Anal. Non Lin\'eaire}, \textbf{20} (2003), no.3, 419-475.

\bibitem[4]{CM08} \newblock S. Cuccagna, T. Mizumachi, \newblock {On asymptotic stability in energy space of ground states for nonlinear Schr{\"o}dinger equations,} \newblock \emph{Comm. Math. Phys.}, \textbf{284} (2008), no.1, 51--77.

\bibitem[5]{FG14} \newblock J. Fr\"ohlich, Z. Gang, \newblock {Emission of Cherenkov radiation as a mechanism for hamiltonian friction,} \newblock \emph{Adv. Math.}, \textbf{264} (2014), 183-235.

\bibitem[6]{FGS10} \newblock J. Fr\"ohlich, Z. Gang: A. Soffer, \newblock {Some hamiltonian models of friction,} \newblock \emph{Jour. Math. Phys.}, \textbf{52} (2011), 083508.   

\bibitem[7]{G15} \newblock Z. Gang, \newblock {A resonance problem in relaxation of ground states of nonlinear Schr\"odinger equaions,} \newblock arXiv: 1505.01107. 

\bibitem[8]{GS07} \newblock Z. Gang, I.-M. Sigal, \newblock {Relaxation of solitons in nonlinear {S}chr{\"o}dinger equations with potential,} \newblock \emph{Adv. Math.}, \textbf{216} (2007), no.2, 443--490.

\bibitem[9]{GW08} \newblock Z. Gang, M.-I. Weinstein, \newblock {Dynamics of nonlinear Schr\"odinger/Gross-Pitaevskii equations; mass transfer in systems with solitons and degenerate neutral modes,} \newblock \emph{Analysis and PDE}, \textbf{1} (2008), no.3, 267-322.

\bibitem[10]{GW11} \newblock Z. Gang, M.-I. Weinstein, \newblock {Equipartition of energy of nonlinear Schr\"odinger equations,} \newblock \emph{Applied Mathematics Research Express}, AMRX 2011, no. 2, 123-181.

\bibitem[11]{KMM} \newblock M. Kowalczyk, Y. Martel, C. Mu\~noz, \newblock {Kink dynamics in the $\phi^4$ model: asymptotic stability for odd perturbations in the energy space,} \newblock \emph{J. Amer. Math. Soc.}, \textbf{30} (2017), 769-798.  

\bibitem[12]{Ko} \newblock A. Komech, \newblock {Attractors of nonlinear Hamiltonian PDEs,} \newblock arXiv: 1409.2009.

\bibitem[13]{L-O-S} \newblock A. Lawrie, S.-J. Oh, S. Shahshahani, \newblock {Gap eigenvalues and asymptotic dynamics of geometric wave equations on hyperbolic space,} \newblock \emph{J. Funct. Anal.} \textbf{271} (2016), no. 11, 3111-3161.

\bibitem[14]{Miz08} \newblock T. Mizumachi, \newblock {Asymptotic stability of small solitary waves to 1D nonlinear Schr{\"o}dinger equations with potential,} \newblock \emph{J. Math. Kyoto Univ.}, \textbf{48} (2008), no.3, 471--497.

\bibitem[15]{S93} \newblock I.-M. Sigal, \newblock {Nonlinear wave and Schr\"odinger equations. I: instability of periodic and quasiperiodic solutions,} \newblock \emph{Comm. Math. Phys.}, \textbf{153} (1993), no.2, 297-320.

\bibitem[16]{SW} \newblock A. Soffer, M.-I. Weinstein, \newblock {Resonances, radiation damping and instability in Hamiltonian nonlinear wave equations,} \newblock \emph{Invent. math.} \textbf{136} (1999), 9-74.

\bibitem[17]{Sof-Wei1} \newblock A. Soffer, M.-I. Weinstein, \newblock  {Multichannel nonlinear scattering for nonintegrable equations,} \newblock \emph{Comm. Math. Phys.}, \textbf{133} (1990), no.1, 119-146.

\bibitem[18] {Sof-Wei2} \newblock A. Soffer, M.-I. Weinstein, \newblock {Selection of the ground state for nonlinear Schr\"odinger equations,} \newblock \emph{Rev. Math. Phys.}, \textbf{16} (2004), no.8, 977-1071.

\bibitem[19]{Sof-Wei3} \newblock A. Soffer, M.-I. Weinstein, \newblock {Theory of nonlinear dispersive waves and selection of the ground state,} \newblock \emph{Phys. Rev. Lett.}, \textbf{95}.21 (2005): 213905.

\bibitem[20] {Ts} \newblock T.-P. Tsai, \newblock {Asymptotic dynamics of nonlinear Schr\"odinger equations with many bound states,} \newblock \emph{Jour. Math. Phys.}, \textbf{192} (2003), no.1, 225-282.

\bibitem[21]{TY} \newblock T.-P. Tsai, H.-T. Yau, \newblock {Relaxation of excited states in nonlinear Schr\"odinger equations,} \newblock \emph{Int. Math. Res. Not.}, \textbf{31} (2002), 1629-1673. 

\end{thebibliography}
\end{document}